\documentclass[12pt, reqno]{amsart}
\usepackage{latexsym, amsmath, amssymb}
\usepackage{mathrsfs}
\usepackage{amsrefs}
\usepackage{graphicx}
\usepackage[all]{xy}
\usepackage[utf8]{inputenc}

\newcommand{\aaa}{\ensuremath{\mathbf{a}}}

\newcommand{\sss}{\ensuremath{\mathbf{s}}}
\newcommand{\xx}{\ensuremath{\mathbf{x}}}
\newcommand{\yy}{\ensuremath{\mathbf{y}}}

\newcommand{\TT}{\ensuremath{\mathcal{T}}}
\newcommand{\LipOp}{\operatorname{\rm{Lip}}}
\newcommand{\sgn}{\operatorname{\rm{sgn}}}
\newcommand{\Id}{\operatorname{\rm{Id}}}

\newcommand{\spn}{\operatorname{\rm{span}}}
\newcommand{\Lip}{\ensuremath{\mathrm{Lip}}}
\newcommand{\Mol}{\ensuremath{\mathrm{Mol}}}
\newcommand{\cop}{\operatorname{\mathrm{co}}_p}

 \newcommand{\Rea}{\mathbb{R}}
 \newcommand{\Nat}{\mathbb{N}}

\newcommand{\Sym}{\mathbb{S}}
\newcommand{\C}{\ensuremath{\mathcal{C}}}
 
 \newcommand{\PP}{\ensuremath{\mathcal{P}}}
 \newcommand{\MM}{\ensuremath{\mathcal{M}}}
 \newcommand{\NN}{\ensuremath{\mathcal{N}}}
 \newcommand{\AAA}{\ensuremath{\mathcal{A}}}
\newcommand{\ee}{\ensuremath{\mathbf{e}}}
\newcommand{\F}{\ensuremath{\mathcal{F}}}
\newcommand{\T}{\ensuremath{\mathcal{T}}}
\newcommand{\R}{\ensuremath{\mathcal{R}}}
\newcommand{\Ss}{\ensuremath{\mathcal{S}}}
\newcommand{\AEM}{\ensuremath{\text{\AE}}}

\allowdisplaybreaks

\newtheorem{Theorem}{Theorem}[section]
\newtheorem{Lemma}[Theorem]{Lemma}
\newtheorem{Proposition}[Theorem]{Proposition}
\newtheorem{Corollary}[Theorem]{Corollary}
\theoremstyle{remark}
\newtheorem{Remark}[Theorem]{Remark}
\newtheorem{Definition}[Theorem]{Definition}
\newtheorem{Example}[Theorem]{Example}
\newtheorem{Question}[Theorem]{Question}



\begin{document}

\title[Lipschitz free $p$-spaces]{Lipschitz free $p$-spaces  for $0<p<1$}

\author[F. Albiac]{Fernando Albiac}
\address{Mathematics Department--InaMat \\ 
Universidad P\'ublica de Navarra\\
Campus de Arrosad\'{i}a\\
Pamplona\\ 
31006 Spain}
\email{fernando.albiac@unavarra.es}

\author[J. L. Ansorena]{Jos\'e L. Ansorena} 
\address{Department of Mathematics and Computer Sciences\\
Universidad de La Rioja\\ 
Logro\~no\\
26004 Spain}
\email{joseluis.ansorena@unirioja.es}

\author[M. C\'uth]{Marek C\'uth} 
\address{Faculty of Mathematics and Physics, Department of Mathematical Analysis\\
Charles University\\ 
186 75 Praha 8\\
Czech Republic}
\email{cuth@karlin.mff.cuni.cz}

\author[M. Doucha]{Michal Doucha} 
\address{Institute of Mathematics\\
 Czech Academy of Sciences\\
115 67 Praha 1\\
Czech Republic}
\email{doucha@math.cas.cz}

\begin{abstract}
This  paper  initiates the study of  the structure of a new class of $p$-Banach spaces,  $0<p<1$, namely the Lipschitz free $p$-spaces (alternatively called Arens-Eells $p$-spaces) $\F_{p}(\MM)$ over $p$-metric spaces. We systematically develop the theory and show that some results hold as in the case of $p=1$, while some new interesting phenomena appear in the  case $0<p<1$ which have no analogue in the classical setting. For the former, we, e.g.,  show that the Lipschitz free $p$-space over a separable ultrametric space is isomorphic to $\ell_{p}$ for all $0<p\le 1$. On the other hand, solving a problem by the first author and N. Kalton, there are metric spaces $\NN\subset \MM$ such that the natural embedding from $\F_p(\NN)$ to $\F_p(\MM)$ is not an isometry.
\end{abstract}

\thanks{Accepted in Israel J. Math. In this updated version we corrected some misprints and moved the result concerning embeddability of $\ell_p$ from the previous version to a new preprint entitled \emph{Embeddability of $\ell_{p}$ and bases in Lipschitz free $p$-spaces for $0<p\leq 1$}, which is available on arXiv.}


\maketitle

\section{Introduction}

\noindent It is safe to say that  most of the research in functional analysis is done in the framework of Banach spaces. While the theory of the geometry of these spaces has evolved very rapidly over the past sixty years, by contrast, the study of the more general case of quasi-Banach spaces has lagged far behind  despite the fact that the first papers in the subject appeared in the early 1940's (\cites{Aoki1942, Day1940}). The neglect of non-locally convex spaces within functional analysis is easily understood. Even when they are complete and metrizable, working with them requires doing without one of the most powerful tools in Banach spaces: the Hahn-Banach theorem and the duality techniques that rely on it. This difficulty in even making the simplest initial steps has led some to regard quasi-Banach spaces as  too challenging and consequently they have been assigned a secondary role in the theory. However, these challenges have been accepted by some researchers and the number of fresh techniques available in this general setting is now increasing (see a summary in \cite{KaltonHandbook}). We emphasize that proving new results in $p$-Banach spaces for $0<p<1$ often provides an alternative proof even for the limit case $p=1$. Hence, quasi-Banach spaces help us appreciate better and also shed new light on regular Banach spaces. Taking into account that more analysts find that quasi-Banach spaces have uses in their research, the task to know more about their structure seems to be urgent and important.

Every family of classical Banach spaces, like the sequence spaces $\ell_{p}$, the function spaces $L_{p}$, the Hardy spaces $H_{p}$, and the Lorentz sequence spaces $d(w,p)$, have a non-locally convex counterpart  corresponding to the values of $0<p<1$.  In this  paper we study Lipschitz free  $p$-spaces over quasimetric spaces. These new class of $p$-Banach spaces, denoted by $\mathcal F_{p}(\mathcal M)$, are an  analogy of the Lipschitz free spaces $\F(\MM)$, whose study has become a very active research field  within  Banach space theory since the appearance in 1999 of the important book \cite{Weaver2018} by Weaver (here we cite the updated second edition) and, more notably, after  the seminal paper \cite{GodefroyKalton2003} by Godefroy and Kalton in 2003.

Lipschitz free $p$-spaces were introduced in \cite{AlbiacKalton2009} with the sole instrumental purpose to build  examples for each $0<p<1$ of two \textit{separable} $p$-Banach spaces which are Lipschitz-isomorphic but fail to be linearly isomorphic. Whether this is possible or not for $p=1$ remains as of today the single most important open problem in the theory of non-linear classification of Banach spaces. However, even though Lipschitz free $p$-spaces were proved to be of substantial utility in functional analysis, the structure of those spaces has not been investigated ever since. Our goal in this paper is to fill this gap in the theory, to encourage further research in this direction, and help those who want to contribute to this widely unexplored  topic.

To that end, after the preliminary Sect.\ ~\ref{Prelim} on the basics in quasimetric and quasi-Banach spaces,  in Sect.\ ~\ref{MetricEnv} we introduce the  notion of metric envelope of a quasimetric space  $\MM$ and relate it to the existence of non-constant Lipschitz maps on $\MM$ as well as to the Banach envelope when $\MM$ is a quasi-Banach space.  In Sect.\ ~\ref{Sec4} we recall the definition of Lipschitz free $p$-space and bring up to light the main differences and setbacks of this theory with respect to the case $p=1$. We also  settle a question  that was raised in \cite{AlbiacKalton2009} and  use molecules and atoms in order to give an alternative  equivalent definition of Lispchitz free $p$-spaces which will be very useful in order to provide examples of Lipschitz free $p$-spaces isometrically isomorphic to $\ell_{p}$ and $L_{p}$ for $0<p<1$. In Sect.\ ~\ref{ultrametricSec} we completely characterise Lipschitz free $p$-spaces over separable ultrametric spaces, showing that for $p\le 1$ they are isomorphic to $\ell_{p}$. 

The most important results are perhaps the ones in Sect.\ \ref{structureSec}, where we study the relation between the subset structure of a quasimetric space $\MM$ and the subspace structure of $\F_{p}(\MM)$. To be precise, for each $p<1$ and each $p<q\le 1$ we provide an example of a subset $\NN$ of a $q$-metric space $\MM$ such that $\F_{p}(\NN)$ is not naturally a subspace of $\F_{p}(\MM)$. This fact evinces a very important    dissimilarity with respect to the case $p=1$ and solves another problem raised in  \cite{AlbiacKalton2009}.  


Throughout this note we use standard terminology and notation in Banach space theory as can be found in \cites{AlbiacKalton2016}. We refer the reader to \cite{Weaver2018} for basic facts on Lipschitz free spaces and some of their uses, and  to \cite{KPR1984}  for background on quasi-Banach spaces. 

\section{Preliminaries}\label{Prelim}
\noindent There are two main goals in this preliminary section. First we review the notion of quasi-metric space along with the related notion of  quasi-Banach space and their main topological features.  Second, we lay out  the notation and terminology used in this article.

\subsection{Quasimetric spaces and Lipschitz maps.} Given an arbitrary nonempty set $\MM$, a \textit{quasimetric} on $\MM$ is a symmetric map   $\rho\colon\MM \times \MM\to [0,\infty)$ such that $\rho(x,y)=0$ if and only if $x=y$, and for some constant $\kappa\ge 1$, $\rho$ satisfies the quasi-triangle inequality   
\begin{equation}\label{quasinormtriang}
\rho(x,z)\le \kappa(\rho(x,y)+\rho(y,z)), \qquad x,y,z\in \MM.
\end{equation} 
The space $(\MM,\rho)$ is then called a \textit{quasimetric space} (see \cite{Heinonen2001}*{p.\ 109}).  
 A quasimetric $\rho$ on a set $\MM$ is said to be a \textit{$p$-metric}, $0<p\le 1$,    if $\rho^{\, p}$ is a metric, i.e., $$\rho^{\, p}(x,y)\le \rho^{\, p}(x,z)+\rho^{\, p}(z,y), \qquad x,y,z\in \MM,$$  in which case we call $(\MM,\rho)$ a \textit{$p$-metric space}.  An analogue of the Aoki-Rolewicz theorem holds in this context (see \cite{Heinonen2001}*{Proposition 14.5}): every quasimetric space can be endowed with an equivalent $p$-metric $\tau$ for some $0<p\le 1$, i.e.,  there is a constant $C=C(\kappa)\ge 1$ such that
\[
C^{-1}\tau(x,y)\le \rho(x,y)\le C\tau(x,y),\qquad x,y\in \MM.
\]

If $(\MM, \rho)$ and $(\NN, \tau)$ are quasimetric spaces we shall say that a map $f\colon \MM\rightarrow \NN$ is \textit{Lipschitz} if there exists a
 constant $C\ge 0$ so that 
\begin{equation}\label{Lipsconsineq}
\tau( f(x),f(y)) \le C \rho(x,y), \qquad x,y\in \MM.\end{equation}
We denote by $\LipOp(f)$   the smallest constant which can play the role of $C$ in the last inequality \eqref{Lipsconsineq}, i.e.,
\[
\LipOp(f)=\sup\left\{\frac{\tau(f(x) f(y))}{\rho(x,y)}\colon x,y \in \MM, x \not=y\right\}\in [0,\infty).
\]
If $f$ is injective, and both $f$ and $f^{-1}$ are Lipschitz, then we say that $f$ is  \textit{bi-Lipschitz} and that $\MM$ Lipschitz-embeds into $\NN$. If there is a bi-Lipschitz map from ${\MM}$ onto ${\NN}$, the spaces $\MM$ and $\NN$ are said to be \textit{Lipschitz isomorphic}.  A map $f$ from a quasimetric space $(\MM,\rho)$ into a quasimetric space $(\NN,\tau)$ is an \textit{isometry} if 
\[
\tau (f(x),f(y)) = \rho(x,y),\qquad x,y\in \MM.
\]
We shall say that  $(\MM, \rho)$ is a \textit{pointed quasimetric space (or a pointed $p$-metric space, or  a pointed metric space)}, if it has a distinguished point that we call  the \textit{origin} and denote $0$. The assumption of an origin is convenient to normalize Lipschitz functions. 

The \textit{Lipschitz dual} of a quasimetric space $(\MM,\rho)$, denoted $\Lip_0(\MM)$, is the (possibly trivial) vector space of all real-valued Lipschitz functions $f$ defined on $\MM$ such that $f(0) = 0$,  endowed with the Lipschitz norm
\[
\Vert f\Vert_{\text {Lip}} =  \sup\left\{  \frac{|f(x) - f(y)|}{\rho(x,y)} \colon  x, y\in \MM, x\not= y\right\}.
\]
It can be readily checked that $(\Lip_0(\MM), \Vert \cdot\Vert_{\text {Lip}})$ is a  Banach space.

\subsection{Quasi-normed spaces and their Banach envelopes.}\label{SectionQB} 
Recall  that a  \textit{quasi-normed space} is a (real) vector space $X$ equipped with a map $\Vert \cdot\Vert_{X}\colon X\to [0, \infty)$ with the properties:
\begin{enumerate}
\item[(i)] $\Vert x\Vert_{X} >0$ for all $x\not=0$,

\item[(ii)] $\Vert \alpha x\Vert_{X}=|\alpha| \Vert x\Vert_{X}$ for all $\alpha\in\Rea$ and all $x\in X$,

\item[(iii)]  there is a constant   $\kappa\ge 1$  so that for all $x$ and $y\in X$ we have 
\begin{equation}\label{eq:modofconc}\Vert x  +y\Vert_{X} \le \kappa(\Vert x\Vert_{X} +\Vert
y\Vert_{X}).\end{equation}
\end{enumerate}
 A  quasi-norm $\Vert\cdot\Vert_{X}$ induces a  linear metric topology. $X$ is called a \textit{quasi-Banach space} if $X$ is complete for this metric. Given $0<p\le 1$, $X$ is said to be  a \textit{$p$-normed} space if the quasi-norm  $\Vert \cdot \Vert_{X}$ verifies (i), (ii) and it is $p$-subadditive, i.e.,
\begin{enumerate}
\item[(iv)] $\Vert x+y\Vert_{X}^{p}\le \Vert x\Vert_{X}^{p} +\Vert y\Vert_{X}^{p}$
for all $x,y\in X$.
\end{enumerate}
Of course, (iv) implies (iii), and, by the Aoki-Rolewicz theorem (see \cite{KPR1984}), we also have that  (iii) implies (iv).  In  the case when $X$ is $p$-normed, a metric inducing the topology can be defined by $d(x,y)= \Vert x-y\Vert_{X}^{p}$.  A quasi-Banach space with an associated $p$-norm is also called a \textit{$p$-Banach space}.

A map $\Vert\cdot\Vert_{X}\colon X\to [0,\infty)$ that verifies properties (ii) and (iv) is called a \textit{$p$-seminorm} on $X$. Given a $p$-seminorm $\Vert\cdot\Vert_{X}$ on a vector space $X$ it is standard to construct a $p$-Banach space from the pair $(X,\Vert\cdot\Vert_{X})$ following the so-called \textit{completion method}. For that we consider the vector subset $N=\{x\in X\colon \Vert x\Vert_{X}=0\}$ and form the quotient space $X/N$, which is $p$-normed when endowed with $\Vert\cdot\Vert_{X}$. Now we just need to complete $(X/N, \Vert\cdot\Vert_{X})$. The reader should be acquainted with the fact that completeness and completion for quasi-metric spaces are completely analogous to such notions for metric spaces.

Given $0<p\le 1$, a subset $\C$ of a vector space $V$ is said to be \textit{absolutely $p$-convex} if for any $x$ and $y\in \C$ and any scalars $\lambda$ and $\mu$  with
$\ |\lambda|^p+|\mu|^p \le 1$ we have $\lambda \, x + \mu\, y \in \C$. The Minkowski functional $\Vert \cdot\Vert_{\C}$  of an absolutely $p$-convex set $\C$, given by 
\[
\Vert x\Vert_{\C}= \inf \left\{\lambda>0 \colon \lambda^{-1}x \in \C\right\},
\]
 defines a $p$-seminorm on $\spn(\C)$.

Given  a nonempty subset $Z$ of a vector space $V$ there is a method  for building a $p$-Banach space from it. Let $\cop(Z)$ denote the  \textit{$p$-convex hull} of $Z$, i.e., the smallest absolutely $p$-convex set containing $Z$. 
If $N=\{x\in \spn(Z)\colon \Vert x\Vert_{\cop(Z)}=0\}$, then the quotient space $\spn(Z)/N$ equipped with $\Vert \cdot \Vert_{\cop(Z)}$ is a $p$-normed linear space.   In the case when $\spn(Z)^{\ast}$ separates the points of $\spn(Z)$ then $\Vert \cdot\Vert_{\cop(Z)}$ is a $p$-norm. 

\begin{Definition}\label{complmeth} The   completion of $(\spn(Z)/N, \Vert \cdot \Vert_{\cop(Z)})$ will be called  the \textit{$p$-Banach space constructed from $Z$  by the $p$-convexification method} and  will be denoted by $(X_{p,Z},\Vert \cdot \Vert_{p,Z})$.
\end{Definition} 

Notice that it is possible to give an explicit expression for $\Vert \cdot \Vert_{p,Z}$. As a matter of fact,  for $x\in X_{p,Z}$ we have
\begin{equation}\label{Minko}
\Vert x\Vert_{p,Z}=\inf\left\{\left(\sum_{j=1}^{\infty}|a_{j}|^{p}\right)^{1/p}\colon x=\sum_{j=1}^{\infty }a_{i} \, x_i, \quad x_i\in Z \right\}.
\end{equation}

When dealing with a quasi-Banach space $X$ it is often convenient to know which is the ``smallest" Banach space containing $X$  or, more generally, given $0<q\le 1$, the smallest $q$-Banach space containing $X$. 

\begin{Definition}\label{BanEnvelnorm} Given a quasi-Banach space $X$ and $0<q\le 1$,
 the \textit{$q$-Banach envelope} of $X$ (resp.\ \textit{Banach envelope} for $q=1$), denoted 
$(\widehat{X^q},\Vert \cdot\Vert_{c,q})$ (resp.\ $(\widehat{X},\Vert \cdot\Vert_{c})$ for $q=1$)
 is the $q$-Banach space  obtained by applying  to the  unit ball $B_X$ of $X$ the $q$-convexification method.
 \end{Definition}

Obviously $\Vert x\Vert_{c,q}\le \Vert x\Vert$ for all $x\in X$, so that the identity map on $X$ induces a (not necessarily one-to-one)  bounded linear map  $i_{X,q}\colon X \to \widehat{X^q}$  whose  range is dense in $\widehat{X^q}$. This map possesses the following universal property: if $T\colon X\to Y$ is a bounded linear map and $Y$ is an arbitrary $q$-Banach space then $T$ factors through $i_{X,q}$, 
 \begin{equation*}
\xymatrix{X \ar[rr]^T \ar[dr]_{i_{X,q}} & & Y\\
& \widehat{X^q}\ar[ur]_{{\widehat T}} &}
\end{equation*}
and the unique  ``extension'' $\widehat T\colon \widehat{X^q}\to Y$  has the same norm as $T$. In particular, $X$ and  $\widehat{X^q}$ have the same dual space.
 
For instance, the $q$-Banach envelope of  $\ell_{p}$  for $0<p<q\le 1$ is $\ell_{q}$.

The following  formula for the $q$-Banach envelope quasi-norm will be very useful. The case $q=1$ was shown by Peetre in \cite{Peetre1974}. 
\begin{Lemma}
Let $X$ be a quasi-Banach space and $0<q\le 1$. Then for $x\in X$,
\begin{equation}\label{BEnorm}
\Vert x \Vert_{c,q}=\inf\left\{ \left(\sum_{i=1}^{n}\Vert x_{i}\Vert^q\right)^{1/q} \colon \sum_{i=1}^{n} x_{i}=x,\, x_{i}\in X,\, n\in \Nat\right\}.
\end{equation}
\end{Lemma}

\begin{proof}
Let $\Vert\cdot\Vert_{0}$ be the $q$-seminorm on $X$ defined by the expression in \eqref{BEnorm} and $X_0$ be the $q$-Banach space obtained from $(X, \Vert\cdot\Vert_{0})$ by the completion method. If $T\colon X \to Y$ is a bounded linear map and $Y$ is $q$-Banach, then $\Vert T(\xx)\Vert \le \Vert T\Vert \Vert \xx\Vert_0$.
Consequently, $X_0$ has the same universal property as $\widehat{X^q}$, thus $X_0$ and $\widehat{X^q}$ are isometric.
 \end{proof}
 
 \subsection{$p$-norming sets in quasi-Banach spaces}

 \begin{Definition} Given a quasi-Banach space $X$ and $0<p\le 1$, we say that a subset $Z$ of $X$ is a \textit{$p$-norming set with constants $C$ and $D$} if
\[
\frac{1}{C} \overline{\cop(Z)} \subseteq B_X \subseteq D  \overline{\cop(Z)}.
\]
In the case when $C=D=1$ we say that $Z$ is \textit{isometrically $p$-norming}.
\end{Definition}

Note that $Z$ is a $p$-norming set of $X$ if and only if $\Vert \cdot\Vert_{p,Z}$ defines an equivalent quasi-norm on $X$. Consequently, if $X$ admits a $p$-norming set then  $X$ is isomorphic to a $p$-Banach space. Conversely, if $X$ is a $p$-Banach space, then a set $Z\subseteq X$ is  $p$-norming with constants $C$ and $D$ if and only if 
\begin{equation}\label{Transformation}
\frac{1}{C} Z \subseteq B_X \subseteq D \overline{\cop(Z)}.
\end{equation}
Adopting the  terminology from harmonic analysis it can  be said that a set $Z$ is $p$-norming in $X$ if and only if $(Z,\ell_p)$ is an atomic decomposition of $X$. Recall that  a  pair $(\AAA,\Sym)$, where $\AAA$ is a subset of $X$ and $\Sym$ is a symmetric sequence space,  is said to be an \textit{atomic decomposition} of $X$
if there are constants $0<C, D<\infty$ such that
\begin{itemize}
\item[(i)] Given $f=(a_n)_{n=1}^\infty\in \Sym$ and $(\alpha_n)_{n=1}^\infty\subset\AAA$ then $\sum_{n=1}^\infty a_n \,\alpha_n$ converges in $X$ to a vector $x$ verifying $\Vert x\Vert  \le C \Vert f\Vert_{\Sym}$, and

\item[(ii)] for any $x\in X$ there are $f=(a_n)_{n=1}^\infty\in\Sym$ and $(\alpha_n)_{n=1}^\infty\subset\AAA$ such that $x= \sum_{n=1}^\infty a_n \,\alpha_n$
and $\Vert f\Vert_{\Sym} \le D \Vert x\Vert$.
\end{itemize}

We conclude this preliminary section enunciating for future reference a few straightforward auxiliary results on $p$-norming sets.
\begin{Lemma}\label{lemma:density} Suppose $Z_1$ and $Z_{2}$ are subsets of a quasi-Banach $X$ such that 
$Z_{1} \subseteq Z_2$, 
$Z_1$ is dense in $Z_2$,  and
$Z_2$ is $p$-norming in $X$.
Then $Z_1$ is a $p$-norming set in $X$ with the same constants as $Z_2$.
\end{Lemma}

\begin{Lemma}\label{lemma:transformation} 
Suppose that $Z_1$ and $Z_2$ are $p$-norming sets for quasi-Banach spaces  $X_1$ and $X_2$, respectively. Let $T$  be a one-to-one linear map from $\spn(Z_1)$ into $X_2$ such that $T(Z_1)=Z_2$. Then $T$ extends to  an onto isomorphism  $\widetilde T\colon X_1 \to X_2$. Moreover, in the case when $Z_1$ and $Z_2$ are both isometrically $p$-norming sets, $\widetilde T$ is an isometry.
\end{Lemma}

\begin{Lemma}\label{ReductionLemma}  Suppose that $Z$ is a $p$-norming set for a quasi-Banach space $X$ with constants $C_1$ and $C_2$ and that  $Z_0\subseteq Z$. If there is a constant $C$  such that  every $x\in Z$ can be written as $x= \sum_{n=1}^\infty a_n \, x_n$ for some $f=(a_n)_{n=1}^\infty\in \ell_p$ with $\Vert f\Vert_p\le C$ and 
$(x_n)_{n=1}^\infty$ in $Z_0$,
then $Z_0$ is  a $p$-norming set for $X$ with constants $C_1$ and $C C_2$. 
 \end{Lemma}
 
\begin{proof}By hypothesis $Z\subseteq C \, \overline{\text{co}_p(Z_0)}$.
Therefore $\overline{\cop(Z)}\subseteq C \, \overline{\cop(Z_0)}$, and so
\[
\frac{1}{C_1} \overline{\cop(Z_0)} \subseteq  \frac{1}{C_1} \overline{\cop(Z)}\subseteq B_X\subseteq  C_2\, \overline{\cop(Z)}\subseteq C C_2 \, \overline{\cop(Z_0)}.\qedhere
\]
 \end{proof}

\section{The metric envelope of a quasimetric space}\label{MetricEnv}

\noindent Suppose $(\MM, \rho)$ is a pointed quasimetric space.  By analogy with the universal extension property of the Banach envelope of a quasi-Banach space, we are interested in the question on how to construct a metric space $(\widetilde{\MM}, \widetilde{\rho})$, and a map $Q\colon \MM\to \widetilde\MM$ with $\LipOp(Q)\le 1$ such that whenever $(M, d)$ is a metric space and $f\colon \MM \to M$   verifies the Lipschitz condition
\begin{equation}\label{Lipsc}
d(f(x),f(y)) \le C \rho(x,y), \qquad x,y\in \MM,
\end{equation}
then $f$ induces a Lipschitz map $\widetilde f\colon\widetilde{\MM}\to M$  with $f=\widetilde f\circ Q$ and  $d(\widetilde f(x),\widetilde f(y)) \le C \widetilde\rho(x,y)$ for all  $x,y\in \widetilde{\MM}.$ 

Note that if $f$ verifies \eqref{Lipsc}, then we will have
\[
d(f(x), f(y))\le C\sum_{i=0}^{n}\rho(x_{i}, x_{i+1}),
\]
for any finite sequence $x=x_{0}, x_{1}, \dots, x_{n+1}=y$ of (possibly repeated) points in $\MM$.
Therefore, in all fairness  we define, for $x,y\in \MM$,
\begin{equation}\label{MetricEnvelope}
\widetilde\rho(x,y)=\inf\sum_{i=0}^{n}\rho(x_{i}, x_{i+1}),
\end{equation}
where the infimum is taken over all  sequences $x=x_{0}, x_{1}, \dots, x_{n+1}=y$ of  finitely-many points in $\MM$. Clearly, $\widetilde\rho$ is symmetric, satisfies the triangle inequality, and does not exceed $\rho$.  Before going on, let us point out that $\widetilde\rho(x,y)$ can be zero for different points $x,y$ in $\MM$.

\begin{Example}\label{trivialmetricspace}  A metric space
$(\MM, d)$ is metrically convex (see  \cite{BenLin2000}) if  for every $x,y\in \MM$ and any $0<\lambda<1$ there exists $z_{\lambda}\in \MM$ with 
\[
d(x,z_{\lambda})=\lambda d(x,y)\quad \text{and}\quad d(y, z_{\lambda})=(1-\lambda)d(x,y).
\]
Let $(\MM, d)$ be  a metrically convex space and, for $0<p<1$, consider the $p$-metric $\rho = d^{1/p}$ on $\MM$. Then $\widetilde\rho(x,y)=0$ for any $x,y\in \MM$. Indeed, given $x\not=y$ in $\MM$, by the metric convexity of $\MM$  for every $n\in \Nat$ we can find a chain of points $\{x_{0},x_{1},\dots, x_{n}\}$ where $x_{0}=x$, $x_{n}=y$, and $d(x_{j-1},x_{j})=d(x,y)/n$ for each $j=1,2,\dots, n$. By  the definition we then have
\[
\widetilde\rho(x,y)\le \left(\frac{d(x,y)}{n}\right)^{1/p}n=\frac{d(x,y)}{n^{1/p-1}}\to 0.
\]
Thus, $\widetilde\rho(x,y)=0$.  
\end{Example}

In view of that, we shall identify points in $\MM$ that are at a zero $\widetilde\rho$-distance, which  leads to the following definition.

\begin{Definition}\label{def:metricenvelope} Let $(\MM,\rho)$ be a quasimetric space and $\tilde\rho$  as in \eqref{MetricEnvelope}. We consider the equivalence relation
\[
x\sim y \Longleftrightarrow \widetilde\rho(x,y)=0,
\]
and define $\widetilde{\MM}$ to be the quotient space $\MM/\sim$.  If $\widetilde x$ and $\widetilde y$ denote the respective equivalence classes of $x$ and $y$, we put  $\tilde\rho(\widetilde x,\tilde y)=\widetilde\rho(x,y)$. The metric space  $(\widetilde{\MM}, {\widetilde \rho})$, together with the quotient map $Q\colon\MM\to\widetilde{\MM}$ will be called the \textit{metric envelope} of  $(\MM,\rho)$.
\end{Definition}

Our discussion yields that the metric envelope of a quasimetric space   is characterized by  the following universal property.

\begin{Theorem}\label{UniversalPropertyMetricEncv} Suppose $(\widetilde \MM,\widetilde \rho,Q)$ is the metric envelope of a quasimetric space  $(\MM,\rho)$. Then:
\begin{enumerate}
\item[(i)] $\LipOp(Q)=1$, and
\item[(ii)] whenever $(M, d)$ is a metric space and $f\colon (\MM,\rho) \to (M,d)$ is $C$-Lipschitz, there is a unique map $\widetilde f\colon (\widetilde{\MM}, \widetilde \rho)\to (M,d)$ such that $f= \widetilde f\circ Q$ is $C$-Lipschitz. Pictorially,
\begin{equation*}
\xymatrix{(\MM,\rho) \ar[rr]^f \ar[dr]_{Q} & & (M, d)\\
& (\widetilde{\MM}, {\widetilde \rho}) \ar[ur]_{{\widetilde f}} &}
\end{equation*}
\end{enumerate}  
\end{Theorem}

\begin{Remark}\label{Lemma1} Theorem~\ref{UniversalPropertyMetricEncv} can be rephrased as saying that for every metric space $(M,d)$ the mapping $g\mapsto Q\circ g$ defines an isometry from  $\Lip_{0}(\widetilde{\MM},M)$ onto $\Lip_{0}(\MM,M)$, and so these two spaces can be naturally identified.
\end{Remark}

Note that, in this language,  Example~\ref{trivialmetricspace} yields that for $0<p<1$, the metric envelope of $\Rea$ equipped with the $p$-metric $\rho(x,y)=|x-y|^{1/p}$ is trivial. On the other hand, by \cite{Albiac2008}*{Lemma 2.7}, $\Lip_0(\mathbb R,\rho)=\{0\}$.   Next we see that this is not a coincidence.

\begin{Proposition}\label{propoor} Given a quasimetric space $(\MM,\rho)$ the following are equivalent.
\begin{itemize} 
\item[$\bullet$] $(\widetilde\MM,\tilde\rho)$ is trivial.

\item[$\bullet$] $\Lip_{0}(\MM, M)=\{0\}$ for any metric space $(M,d)$.
\item[$\bullet$] $\Lip_0(\MM)=\{0\}$.
\end{itemize}
\end{Proposition}

\begin{proof} If $(\widetilde\MM,\tilde\rho)$ is trivial it is clear that $\Lip_0(\widetilde \MM, M)=\{0\}$ for any metric space $M$. Using Remark~\ref{Lemma1} we get $\Lip_0(\MM, M)=\{0\}$. 

If  $\Lip_{0}(\MM, M)=\{0\}$ for any metric space $M$  in particular it holds for $M=\Rea$, i.e.,   $\Lip_0(\MM)=\{0\}$.

Finally, if  $(\widetilde\MM,\tilde\rho)$ is non-trivial then clearly $\Lip_0(\widetilde\MM)$ is non-trivial and so by Remark~\ref{Lemma1} we get $\Lip_0(\MM)\not=\{0\}$.
\end{proof}

\begin{Example}\label{metricenvelopeLp} Let $0<p<1$. We know that the $p$-metric space $L_{p}[0,1]$ equipped with the usual $p$-metric induced by the $p$-norm, given by 
\[\rho(f,g)=\Vert f-g\Vert_{p},\quad f, g\in L_{p}[0,1]\] has $\Lip_0(L_p[0,1])=\{0\}$ (see \cite{Albiac2008}*{Proposition 2.8}). Then, by Proposition~\ref{propoor} we infer that its metric envelope is trivial.
\end{Example}

Let us next show that the fact $\widetilde{L_{p}[0,1]}=\{0\}$ is related to the well-known property that the Banach envelope of the $p$-Banach space $L_{p}[0,1]$ for $0<p<1$ is trivial. In fact, metric  and Banach envelopes are related by the following result.
\begin{Proposition}\label{prop:metricvsnormed} Let $(X,\Vert \cdot\Vert)$ be a $p$-normed space. Consider on $X$ the $p$-metric  $\rho$  given by $\rho(x,y)=\Vert x-y\Vert$ and let $0$ be the distinguished point of $X$. Then $\widetilde\rho(x,y)=\Vert x-y \Vert_c$ for all $x,y\in X$, where $\Vert\cdot\Vert_{c}$ is the norm introduced in Definition~\ref{BanEnvelnorm}.
\end{Proposition}

\begin{proof}The set of all tuples $(y_j)_{j=0}^n$ with $y_0=x$  and $y_n=y$  coincides with the set of all tuples of the form $(x+\sum_{k=0}^j x_k)_{j=0}^n$, where $x_0=0$ and $\sum_{j=1}^n x_j=y-x$. Hence,
\begin{align*}
\widetilde\rho(x,y)
&=\inf\left\{ \sum_{j=1}^n \rho\left(x+\sum_{k=1}^j x_k, x+\sum_{k=1}^{j-1} x_k\right) \colon \sum_{j=1}^n x_j=y-x\right\}\\
&=\inf\left\{ \sum_{j=1}^n \Vert x_j\Vert \colon \sum_{j=1}^n x_j=y-x\right\}
=\Vert y-x\Vert_c.
\qedhere\end{align*}
\end{proof}

\begin{Remark} Note that it is possible to extend Definition~\ref{def:metricenvelope} and Theorem~\ref{UniversalPropertyMetricEncv}  to the case when $0<q<1$. Indeed, given a pointed quasimetric space $(\MM,\rho)$  we  define its \textit{$q$-metric envelope} $(\widetilde{\MM^{q}}, \widetilde{\rho_{q}})$ following the same steps as in the construction of its metric envelope. The $q$-metric $\widetilde{\rho_{q}}$ is given by
\[
\widetilde\rho_{q}(x,y)=\inf\left(\sum_{i=0}^{n}\rho^{q}(x_{i}, x_{i+1})\right)^{1/q},\qquad x,y\in \MM,
\]
the infimum being taken over all finite sequences $x=x_{0}, x_{1}, \dots, x_{n+1}=y$ of points in $\MM$, and $\widetilde{\MM^{q}}$ is the quotient space $\MM/\sim_{q}$. The equivalence relation here is the expected one, i.e.,
 \[
x\sim_{q} y \Longleftrightarrow \widetilde\rho_{q}(x,y)=0.
\]
Thus $(\widetilde{\MM^{q}}, \widetilde{\rho_{q}})$  is the   pointed $q$-metric space (having as a distinguished point the equivalence class of $0$) characterized by  the following universal property:  whenever $(M, d)$ is a $q$-metric space and $f\colon (\MM,\rho) \to (M,d)$ is $C$-Lipschitz then the map $\widetilde f\colon (\widetilde{\MM^{q}}, \widetilde{\rho_{q}}) \to (M,d)$ such that $f= \widetilde f\circ Q$ is $C$-Lipschitz, where $Q\colon \MM\to \widetilde{\MM^{q}}$ denotes the canonical quotient map:
\[
\xymatrix{(\MM,\rho) \ar[rr]^f \ar[dr]_{Q} & & (M, d)\\
&(\widetilde{\MM^{q}}, \widetilde{\rho_{q}}) \ar[ur]_{{\widetilde f}} &}
\]

Notice also  that if we regard a $p$-Banach space $(X,\Vert\cdot\Vert_{X})$   as a pointed $p$-metric space in the obvious way (i.e., by taking $0$ as the origin of the vector space $X$ equipped with the $p$-metric $\rho(x,y)=\|x-y\|_{X}$),  Proposition~\ref{prop:metricvsnormed} can be generalized as well, and we can come to the conclusion that the  $q$-Banach envelope of $X$ is the completion of its $q$-metric envelope. We leave out for the reader to check the straightforward details.
\end{Remark}

\section{Lipschitz free $p$-spaces over quasimetric spaces}\label{Sec4}

\noindent Every metric space embeds isometrically into a Banach space. Similarly, the natural environment to isometrically embed quasimetric spaces   will be $p$-Banach spaces.
Notice that for  $0<p<1$, every pointed $p$-metric space $\MM$  embeds isometrically into a ``huge'' $p$-Banach space, namely the space $Y=\ell_{\infty}(\MM;L_{p}(0,\infty))$ of bounded functions from $\MM$ into the real space $(L_{p}(0,\infty),\Vert\cdot\Vert_{p})$ endowed with the  $p$-norm
\[
\Vert f\Vert_{Y}=\sup_{x\in \MM}\Vert f(x)\Vert_{p}.
\]
Indeed, with the convention that $\chi_{(a,b]}=-\chi_{(b,a]}$ if $a>b$, the map $\Psi\colon\MM\to \ell_{\infty}(\MM;L_{p}(0,\infty))$ given by
$
\Psi(x)=\left(\chi_{(\rho^{\, p}(0,y),\rho^{\, p}(x,y)]}\right)_{y\in\MM}
$
does  the job (see \cite{AlbiacKalton2009}*{Proposition 3.3}).
Of course, depending on the $p$-metric space we can find  simpler  (isometric) embeddings, like the map 
 \begin{equation}\label{eq:embedLp}
 \Phi\colon (\Rea,|\cdot|^{1/p})\to L_p(\Rea), \quad
 \Phi(x)=\chi_{(0,x]}.
 \end{equation}

Once we have  accomplished the task to  embed  a $p$-metric space $\MM$ into a $p$-Banach space, it seems natural to look for an ``optimal'' way to do it, in the sense that  every Lipschitz map from $\MM$ into a $p$-Banach space factors through it. The following construction from \cite{AlbiacKalton2009} attains this goal.

Let $\Rea_{0}^{\MM}$ be the space of all (not necessarily continuous) maps  $f\colon \MM\to\Rea$ so that $f(0)=0$ and let $\PP(\MM)$  be the linear span in the linear dual $(\Rea_{0}^{\MM})^{\#}$ of the evaluations $\delta(x)$, where $x$ runs through $\MM$, defined by
\begin{equation}\label{Delta}
\langle \delta(x),f\rangle=f(x), \qquad f\in \Rea_{0}^{\MM}.
\end{equation}
 Note that $\delta(0)=0.$ 

If $\mu=\sum_{j=1}^{N}a_j\delta(x_j)\in\PP(\MM)$, put
\begin{equation}\label{definition}
 \|\mu\|_{_{\F_{p}(\MM)}}=\sup \left\|\sum_{j=1}^N a_jf(x_j)\right\|_{Y},
 \end{equation}
the supremum being taken over all $p$-normed spaces $(Y,\Vert\cdot\Vert_{Y})$ and all $1$-Lipschitz maps $f\colon \MM\to Y$ with $f(0)=0$. It is straightforward to check that formula \eqref{definition} defines a $p$-seminorm on $\PP(\MM)$. In fact, the following proposition shows that $ \|\cdot\|_{_{\F_p(\MM)}}$ is a $p$-norm, thus settling a question posed in \cite{AlbiacKalton2009}.

\begin{Proposition}\label{linindep} 
Let  $({\MM},\rho)$ be a pointed $p$-metric space, $0<p\le 1$. Then $(\PP(\MM),\|\cdot\|_{_{\F_{p}(\MM)}})$ is a $p$-normed space.
\end{Proposition}
\begin{proof} 
Suppose that  $\Vert \sum_{j=1}^N a_j \delta(x_j)\Vert_{_{\F_{p}(\MM)}}=0$  for some  $(a_j)_{j=1}^{N}$  scalars and some  $(x_j)_{j=1}^{N}$ in ${\MM}\setminus\{0\}$. Then  $\sum_{j=1}^N a_j f(x_j)=0$ for every $p$-Banach space $X$ and every Lipschitz map $f\colon {\MM}\to X$ with $f(0)=0$. 

Pick $i\in\{1,\dots,N \}$ and for the sake of convenience denote the distinguished point  of $\MM$  by $x_0$. Since the set $\NN=\{x_j \colon 0\le j\le N\}$ is finite, the map
from the metric space $(\NN,\rho^{\, p})$ into $(\Rea,|\cdot|)$ given by $x_i\mapsto 1$ and $x_j\mapsto 0$ for $j\not=i$ is  Lipschitz. By McShane's theorem, it extends to a  Lipschitz map $g$ from $(\MM,\rho^{\, p})$ into $(\Rea,|\cdot|)$. In other words, the map $g$ is Lipschitz  from $(\MM,\rho)$ into $(\Rea,|\cdot|^{1/p})$. If $\Phi$ is as in \eqref{eq:embedLp}, then
 $f:=\Phi\circ g\colon\MM\to L_p(\Rea)$ is  Lipschitz as well. Since $f(0)=0$, we infer that $a_i \chi_{(0,1]}=\sum_{j=1}^N a_j f(x_j)=0$. Hence, $a_i=0$.
\end{proof}

\begin{Definition}[cf.\ \cite{AlbiacKalton2009}] Given a $p$-metric space $\MM$, the  \textit{Lipschitz free $p$-space over $\MM$}, denoted by $\F_{p}(\MM)$,  is the $p$-Banach space resulting from the completion of  the $p$-normed space $(\PP(\MM), \Vert\cdot\Vert_{_{\F_{p}(\MM)}})$. We will refer to the map $\delta_\MM\colon \MM \to \F_p(\MM)$ given by  $\delta_\MM(x)=\delta(x)$ as  the natural embedding of $\MM$ into  $\F_p(\MM)$.
\end{Definition}

In \cite{Kalton2004}, Kalton uses the symbol $\mathcal F_{\omega}(\mathcal M)$ to denote Lipschitz-free Banach spaces associated with metric spaces equipped with distances $\omega\circ d$ that arise after snowflaking, where $\omega$ is a gauge. This is of course different from what is considered in the present work, but we want the reader to be warned to avoid possible confusions. Note  that our considerations are also independent of the work of Petitjean in \cite{Petitjean2017}, where he studies Lipschitz-free spaces over metric spaces induced by $p$-norms.

\begin{Remark} The choice of a base point  in $\MM$ is not relevant in the definition of $\F_{p}(\MM)$.  Indeed, if we change the origin in $\MM$ and apply the construction, 
 we have a natural linear isometry between the resulting Lipschitz free $p$-spaces.\end{Remark}

For expositional ease and further reference, let us point out the following easy consequence of the proof of Proposition~\ref{linindep}.
\begin{Lemma}\label{lem:oodim}Let  $({\MM},\rho)$ be an infinite $p$-metric space, $0<p\le 1$. Then $\F_p(\MM)$ is infinite dimensional.
\end{Lemma}

Similarly to Lipschitz free Banach spaces over metric spaces, the spaces $\F_{p}(\MM)$ for $0<p<1$ are uniquely characterised by the universal property included in the following result from  \cite{AlbiacKalton2009}.

\begin{Theorem}\label{part-case} Let $(\MM,\rho)$ be a pointed $p$-metric  space. Then: 
\begin{enumerate}
\item[(a)]   $\delta_{\MM}$ is an isometric embedding.
\item[(b)] The linear span of $\{\delta_{\MM}(x)\colon x\in \MM\}$ is dense in $\F_{p}(\MM)$.
\item[(c)]   $\F_{p}(\MM)$ is the unique (up to isometric isomorphism) $p$-Banach space  such that for every $p$-Banach space $X$ and every Lipschitz map $f\colon\MM\rightarrow X$ with $f(0) = 0$ there
exists a unique  linear map $T_{f}\colon\F_{p}(\MM)\rightarrow X$ with $T_{f}  \circ \delta_{\MM} = f$. Moreover $\Vert T_{f} \Vert =\LipOp(f)$. Pictorially,
\[
\xymatrix{\MM \ar[rr]^f \ar[dr]_{\delta_{\MM}} & & X\\
& \F_{p}(\MM) \ar[ur]_{T_{f}} &}
\]
\end{enumerate}
\end{Theorem}

\begin{Corollary}\label{cor:sep}The space $\F_{p}(\MM)$ is separable whenever $\MM$ is.
\end{Corollary}

\begin{proof}Note that the map $\delta\colon \MM\to \F_{p}(\MM)$ is an isometric embedding and that $\F_{p}(\MM)$ is the closed linear span of $\delta(\MM)$.
\end{proof}

\begin{Remark}	\label{afterdef} If $p=1$ (so that $\rho$ is a metric) then it follows from the Hahn-Banach theorem that $\F_1(\MM)$ is the space denoted by $\F(\MM)$ in \cites{GodefroyKalton2003, Kalton2004} and  the norm of $\mu=\sum_{j=1}^{N}a_{j}x_{j}\in \PP(\MM)$ can be computed as
 \[
 \|\mu\|_{_{\F_{1}(\MM)}}=\sup \left|\sum_{j=1}^N a_jf(x_j)\right|,\]
the supremum being taken over all 1-Lipschitz maps $f\colon \MM\to \Rea$ with $f(0)=0$. 
Moreover, it is known (see, e.g., \cite{Weaver2018}) that $\F_{1}(\MM)^{\ast}=\Lip_{0}(\MM)$. We advance that the corresponding result also holds for $p<1$, i.e., $\F_{p}(\MM)^{\ast}=\Lip_0(\MM)$.  We will prove this in Corollary~\ref{dualAEp}.
\end{Remark}

Lipschitz free $p$-spaces provide a canonical linearization process of Lipschitz maps between $p$-metric spaces: if we identify (through the map $\delta_{\MM}$) a $p$-metric space $\MM$ with a subset of $\F_{p}(\MM)$, then any Lipschitz map from a $p$-metric space $\MM_{1}$ to a $p$-metric space $\MM_{2}$ which maps $0$ to $0$ extends to a continuous linear map from $\F_{p}(\MM_{1})$ to $\F_{p}(\MM_{2})$. That is:

\begin{Lemma}[cf.\ \cite{GodefroyKalton2003}*{Lemma 2.2}]\label{key} Let $\MM_{1}$ and $\MM_{2}$ be pointed $p$-metric spaces $(0<p\le 1)$ and suppose $f\colon \MM_{1}\rightarrow \MM_{2}$ is a Lipschitz map such that $f(0) = 0$. Then  there exists a unique
linear operator
$L_{f}\colon\F_{p}(\MM_{1})\rightarrow \F_{p}(\MM_{2})$ such that $ L_{f}\delta_{\MM_{1}} = \delta_{\MM_{2}} f$, i.e., the following diagram commutes
\[
\xymatrix{\MM_{1}\ar[r]^f \ar[d]_{\delta_{\MM_{1}}} & \MM_{2}\ar[d]^{\delta_{\MM_{2}}}\\
\F_{p}(\MM_{1}) \ar[r]^{L_{f}}&  \F_{p}(\MM_{2})}
\]
and $\Vert L_{f}\Vert =\LipOp(f)$.  In particular, if $f$ is a bi-Lipschitz bijection then $L_{f}$ is an isomorphism.
\end{Lemma}

\begin{proof} Since $\delta_{\mathcal M_2}$ is an isometric embedding, the map $g:=\delta_{\mathcal M_2}\circ f$ is  Lipschitz with $g(0)=0$ and Lip$(g)=$Lip$(f)$. Now the result follows from Theorem~\ref{part-case}.
 \end{proof}

\subsection{Molecules and atomic decompositions}\label{Sect5}

Given a set $\MM$ and $x\in \MM$, let  $\chi_{x}$ denote the indicator function of the singleton set $\{x\}$. Now, for  $x$ and $y\in {\MM}$
we put 
\[
m_{x,y}:=\chi_x -\chi_y.
\]
Let  $(\MM,\rho)$ be a $p$-metric space for some $0<p\le 1$. A \textit{molecule} of $\MM$ is a function $m\colon \MM\to \Rea$ that is supported on a finite subset of $\MM$ and that satisfies $\sum_{x\in \MM} m(x)=0$. The vector space  of all molecules of a metric space $\MM$ will be denoted by $\Mol({\MM})$. 
  
A simple induction argument shows that every molecule has at least one expression as a linear combination of molecules of the form $m_{x,y}$, so that  $\Mol({\MM})$ coincides with  the linear span of the family of molecules 
\[
\AAA'({\MM})=\left\{\frac{m_{x,y}}{\rho(x,y)} \colon x,y\in {\MM}, x\not=y\right\}\subseteq \Rea^{\MM}.
\]
 
\begin{Definition}We define the \textit{Arens-Eells $p$-space over $\MM$}, denoted $\AEM_p(\MM)$, as the $p$-Banach space constructed from the set $\AAA'({\MM})$ using the $p$-convexification method (see Definition~\ref{complmeth}).
\end{Definition}

This way,  if we give $\Mol({\MM})$ the $p$-seminorm
\begin{equation}\label{formulapsemi}
\Vert m\Vert_{\AEM_{p}}=\inf\left\{\left(\sum_{i=1}^{N}|a_{i}|^{p} \right)^{1/p} :\, m=\sum_{i=1}^{N}a_{i}\frac{m_{x_{i},y_{i}}}{\rho(x_{i},y_{i})},\, N\in \Nat\right\},
\end{equation}
we have that $\AEM_{p}(\MM)$ is the completion of $\Mol({\MM})$ (a priori, modulo the set of molecules with zero $p$-seminorm) with respect to $\Vert\cdot\Vert_{\AEM_{p}}$. However, as we will see below, formula \eqref{formulapsemi} defines in fact a $p$-norm on  $\Mol({\MM})$.

The following result establishes that the Arens-Eells $p$-space over $\MM$ can be identified with  the Lipschitz free $p$-space over $\MM$.  
 \begin{Theorem}\label{AtomicAEp} Let $0<p\le 1$ and $(\MM,\rho)$ be a pointed $p$-metric space. Then $\F_{p}(\MM)$ and $\AEM_p({\MM})$  are isometrically isomorphic. In fact, there is a linear onto isometry $T\colon \F_{p}(\MM)\to\AEM_p({\MM})$
 such that  $T(\delta(x)) = \chi_x -\chi_0$ $=m_{x,0}$ for all $x\in \MM$.
 \end{Theorem}
 \begin{proof}Consider the map $f\colon {\MM}\to \AEM_p({\MM})$ given by $f(x)=m_{x,0}$ for  $x\in{\MM}$. Clearly,  $f(0)=0$ and $f(x)-f(y)=m_{x,y}$ for all $x$, $y\in \MM$. Since $f$ is $1$-Lipschitz, Theorem~\ref{part-case} yields a norm-one linear map $T_{f}\colon \F_p({\MM})\to\AEM_p({\MM})$ such that $T_{f}(\delta(x))=m_{x,0}$.
 
 Since $(\chi_x)_{x\in \MM}$ is a linearly independent family in $\Rea^{\MM}$, there is a linear map from $\spn\{\chi_x \colon x\in \MM\}$ into ${\PP}({\MM})$ that takes $\chi_x$ to $\delta(x)$ for every $x\in \MM$. Let $S_1$ be its restriction to $\Mol({\MM})$.  For $x,y\in {\MM}$ with $x\not=y$ we have
\[
S_1\left(\frac{m_{x,y}}{\rho(x,y)}\right)=\frac{\delta(x)-\delta(y)}{\rho(x,y)},
\]
and
\[
\left\Vert \frac{\delta(x)-\delta(y)}{\rho(x,y)}\right\Vert_{{\F_{p}(\MM)}}\le 1,
\]
so by density $S_1$ extends to a norm-one operator $S$ from  $\AEM_p({\MM})$ into $\F_p({\MM})$. Since $T(S(m))=m$ for every molecule $m$,
and   $S(T(\mu))=\mu$ for every $\mu \in {\PP}({\MM})$, by continuity and density it follows that $T\circ S=\Id_{\AEM_p({\MM})}$ and $S\circ T=\Id_{\F_p({\MM})}$.
\end{proof}
 
The following two results are re-formulations of Theorem~\ref{AtomicAEp}. While the expression of the norm in \eqref{definition} relies on extraneous ingredients, Corollary~\ref{alternative-norm}  provides an intrinsic formula for the $p$-norm on $\F_{p}(\MM)$, i.e., an expression that relies only on the quasimetric on the  space $\MM$.
 
\begin{Corollary}\label{pNormingset} Let $(\MM,\rho)$ be a pointed $p$-metric space, $0<p\le 1$.  The subset of $\PP(\MM)$ given by
\[
\AAA(\MM)=\left\{ \frac{\delta(y)-\delta(x)}{\rho(x,y)} \colon x,y\in \MM, \, x\not=y\right\}
\]
is isometrically $p$-norming for  $\F_{p}(\MM)$.
\end{Corollary}
 
\begin{Corollary}\label{alternative-norm} Let $(\MM,\rho)$ be a pointed $p$-metric space, $0<p\le 1$. 
For  $\mu \in  \F_{p}(\MM)$ we have
\[
\Vert  \mu \Vert_{_{\F_{p}(\MM)}}
= \inf\left\{
\left(\sum_{k=1}^{\infty} |a_{k}|^{p}\right)^{1/p} \colon \mu = \sum_{k=1}^{\infty}a_{k}\frac{\delta(x_{k}) - \delta(y_{k})}{\rho(x_{k}, y_{k})}
\right\}.
\]
\end{Corollary}

\subsection{Applications: early examples and results}

Next we  use Corollary~\ref{pNormingset} to identify the first examples of Lipschitz-free $p$-spaces over quasimetric spaces for $0<p<1$.   The informed reader will see a relation between  the map considered in \eqref{lasthopefully} and the (quite forgotten) theory of flat spaces, developed by J.J. Sch\"affer and others in the 1970' (see e.g. \cite{JJS1967, HK1970}).

\begin{Theorem}\label{LpasFreespace} Let $0<p\le 1$. Let $I$ be an interval of $\mathbb R$ equipped with the $p$-metric $\rho(x,y)=|x-y|^{1/p}$ for $x,y\in I$. Then 
$\F_p(I) \approx L_p(I)$
 isometrically. To be precise, if $a$ is the base point of $I$, the map 
\begin{equation}\label{lasthopefully}
\F_{p}(I) \to L_p(I), \quad 
\delta_{I}(x)\mapsto \chi_{(a,x]}
\end{equation}
extends to a linear isometry.
\end{Theorem}
 
\begin{proof}Choose an arbitrary $a\in I$ as the base point of  $(I,\vert \cdot\vert^{1/p})$. Set 
 \[
 \AAA_{p,I}=\left\{ \frac{\chi_{(x,y]}}{|y-x|^{1/p}} \colon x,y\in I, \, x<y\right\},
 \]
 and let $T\colon \PP(I) \to \Rea^I$ be the  linear map determined  by $
 \delta(x) \mapsto \chi_{(a,x]}$ for
 $x\in I\setminus \{a\}.$
Using the notation of Corollary~\ref{pNormingset}, we put
\[
\AAA(I)=\left\{\frac{\delta(x)-\delta(y)}{|x-y|^{1/p}} : x,y\in I, x\not=y \right\}.
\]
Since $T(\AAA(I))=\{ \pm f \colon f\in  \AAA_{p,I}\}$, taking into account Corollary~\ref{pNormingset}  and Lemma~\ref{lemma:transformation}, it suffices to show that  $\AAA_{p,I}$ is an isometric $p$-norming set for $L_p(I)$. To that end we need to verify that $f\in\overline{\cop}(\AAA_{p,I})$ for every $f\in L_p(I)$ with $\|f\|_p \leq 1$. By density it is sufficient to prove it for step functions. Let $f\colon I\to \Rea$ be a step function, i.e.,  
\[
f=\sum_{j=1}^N a_j \chi_{(x_{j-1},x_j]},
\]
 for some $x_0<x_1<\cdots<x_{j-1}<x_j<\cdots<x_N$ in $I$ and some scalars $(a_j)_{j=1}^N$.
Then, if $b_j= (x_j-x_{j-1})^{1/p} a_j$, we have $f=\sum_{j=1}^N b_j \, f_j$ with $f_j\in  \AAA_{p,I}$  and 
\[
\sum_{n=1}^N |b_j|^p=\sum_{j=1}^n |a_j|^p (x_j-x_{j-1})=\Vert f\Vert_p^p. \qedhere
\] 
\end{proof}

Recall that a quasimetric space $\MM$ is  \textit{uniformly separated} if
\[
\inf\{ \rho(x,y) \colon x,y \in \MM, x\not= y\}>0 .
\]
Let us note that each bounded and uniformly separated quasimetric space is Lipschitz isomorphic to the \textit{$\{0,1\}$-metric space}, i.e., the  metric space whose distance attains only the values $0$ and $1$.
 
\begin{Theorem}\label{discreteA} Let  $\MM$ be a bounded and uniformly separated quasimetric space. For $0<p\le 1$ we have
$\F_p(\MM)\approx \ell_p(\MM\setminus\{0\}).$
To be precise, the map 
\[
\F_{p}(\MM)\to \ell_p (\MM\setminus\{0\}), \quad
\delta_{\MM}(x) \mapsto \ee_x
\]
where $\ee_x$ denotes the indicator function of the singleton $\{x\}$, extends to a linear  isomorphism.

\end{Theorem}
\begin{proof} Without loss of generality we can assume that $(\MM,\rho)$ is the $\{0,1\}$-metric space.
If  $x$ and $y$ are two different points in ${\MM}\setminus\{0\}$ we can write
\[
\frac{\delta(y)-\delta(x)}{\rho(x,y)}=a_{xy}\frac{\delta(x)}{\rho(0,x)} +b_{x,y} \frac{\delta(y)}{\rho(0,y)},
 \]
where $a_{x,y}=  -1$ and  $b_{x,y} =1$.
Since $|a_{x,y}|^p+|b_{x,y}|^p=2$, by Corollary~\ref{pNormingset} and Lemma~\ref{ReductionLemma}, the set
\[
\AAA=\left\{ \frac{ \delta(x)}{\rho(0,x)} \colon x\in \MM\setminus\{0\}\right\}=\left\{ \delta(x) \colon x\in \MM\setminus\{0\}\right\},
\]
 is  $p$-norming for $\F_p(\MM)$ with constants $1$ and $2^{1/p}$.
Consider the linear map
 $T\colon \PP(\MM) \to \Rea^{{\MM}\setminus\{0\}}$ given by $
 \delta(x)\mapsto   \ee_x$. 
 We have that $T(\AAA)= \AAA(\MM\setminus\{0\}):=
 \{\ee_x \colon x\in {\MM}\setminus\{0\}\}$.
Since $\AAA(\MM\setminus\{0\})$ is an isometrically $p$-norming set for  $\ell_p(\MM\setminus\{0\})$,
 Lemma~\ref{lemma:transformation} finishes the proof.
\end{proof}

Notice that,  quantitatively, the proof of Theorem~\ref{discreteA}  gives that if $\MM$ is equipped with the $\{0,1\}$-metric, then
\[
2^{-1/p}\left(\sum_{x\in\MM\setminus\{0\}} |a_x|^p\right)^{1/p} \le \left\Vert \sum_{x\in\MM\setminus\{0\}} a_x\, \delta(x) \right\Vert_{\F_p(\MM)}
\le \left(\sum_{x\in\MM\setminus\{0\}} |a_x|^p\right)^{1/p}
\]
for all scalars $(a_x)_{x\in\MM\setminus\{0\}}$ eventually null. Going further we are going to be able to compute the quasi-norms $\left\Vert \sum_{x\in\MM\setminus\{0\}} a_x\, \delta(x) \right\Vert_{\F_p(\MM)}$ in the case when $a_x\ge 0$. Our argument relies on the construction of a suitable $d$-dimensional absolutely $p$-convex body  for every $d\in\Nat$.

\begin{Proposition}\label{prop:pbody} For every $d\in\Nat$ and every $0<p\leq 1$, there is a $p$-norm $\Vert \cdot\Vert_{(p)}$ on $\Rea^d$ such that:
\begin{enumerate}
\item[(a)] $\Vert (x_j)_{j=1}^n\Vert_{(p)}=(\sum_{j=1}^d x_j^p)^{1/p}$ if $x_j\ge 0$ for $j\in\{1,\dots,d\}$, and 
\item[(b)] $\Vert \ee_i -\ee_j\Vert_{(p)}\le 1$ for all $i$, $j\in \{1,\dots,d\}$.
\end{enumerate}
\end{Proposition}

\begin{proof} Given a vector space $V$ and $Z\subseteq V$, set
\[
\cop^+(Z)=\left\{ \sum_{j=1}^k \lambda_j \, v_j \colon k\in\Nat, \,  \lambda_j\ge 0, \, \sum_{j=1}^k \lambda_j^p \le 1, v_j\in Z\right\}.
\]
For $d\in\Nat$, put $\Nat[d]=\{ 1,\dots,  d\}$. Given $\sss=(s_j)_{j=1}^d\in\Rea^d$, we let $M_\sss$ be the endomorphism of $\Rea^d$ given by
\[
M_\sss((x_j)_{j=1}^d)=(s_j \, x_j)_{j=1}^d.
\]
Given $A\subseteq\Nat[d]$, we put $M_A=M_\sss$  where $\sss=(s_j)_{j=1}^n$ is defined by $s_j=1$ for  $j\in\Nat[d] \setminus A$, and $s_j=-1$ for $j\in A$; that is,  $M_A$ is 
 the symmetry with respect to the subspace $\{ (x_j)_{j=1}^d \in \Rea^d \colon  x_j=0\;\text{for all }   j\in A\}$. Denote
\begin{align*}
\Rea^d_+&=\left \{ (x_j)_{j=1}^d\in\Rea^d \colon x_j\ge 0 \text{ for all } j\in \Nat[d] \right\},\\
B_p&=\left\{ (x_j)_{j=1}^d\in\Rea^d \colon \sum_{j=1}^d |x_j|^p\le 1 \right\},\\
B_\infty&= \left\{ (x_j)_{j=1}^d\in\Rea^d \colon  |x_j|\le 1 \text{ for all } j\in\Nat[d] \right\}.
\end{align*}
Given $i,j\in \Nat[d]$ with $i\not=j$, we define $Z_{i,j}\subseteq\Rea_+^d$ by 
\[
Z_{i,j}=\{a \, \ee_i + b	\,  \ee_j \colon 0\le a, b\le 1\}.
\]
Given disjoint sets $A$, $B\subseteq\Nat[d]$ we define  $\TT_{A,B}\subseteq\Rea^d_+$ by 
\[
\TT_{A,B}=\begin{cases}
\{ 0\}  & \text{ if } A=B=\emptyset\\
\cop^+( \{ \ee_i \colon i\in A\}) ) & \text{ if } A\not=\emptyset \text{ and } B=\emptyset,  \\
\cop^+(\{ \ee_j \colon j\in B\})  & \text{ if } A=\emptyset \text{ and } B \not=\emptyset,\\ 
\cop^+(\cup_{ (i,j)  \in A\times B} Z_{i,j})  & \text{ otherwise.}
\end{cases}
\]
 It is routine to prove that  the the family of bodies $(\TT_{A,B})$ enjoys  the following properties.
\begin{itemize}
\item[(a1)] If $\lambda$, $\mu\ge 0$ are such that $\lambda^p +\mu^p\le 1$, 
then $\lambda\, \TT_{A,B}+\mu\, \TT_{A,B} \subseteq \TT_{A,B}$ for all disjoint $A,B\subseteq \Nat[d]$.
\item[(a2)] $B_p \cap \Rea^d_+ \subseteq \TT_{\Nat[d]\setminus A,A}\subseteq B_\infty$, for all $A\subseteq \Nat[d]$.
\item[(a3)] $B_p \cap \Rea^d_+=\TT_{\Nat[d],\emptyset}$.
\item[(a4)] If $A\subseteq A_1$ and $B\subseteq B_1$, then $\TT_{A,B}\subseteq \TT_{A_1,B_1}$.
\item [(a5)] If $\xx=(x_j)_{j=1}^d \in \TT_{A,B}$ and $x_j=0$ for every $j\notin D$, where $D\subseteq \Nat[d]$, then $\xx\in \TT_{A\cap D,B\cap D}$.
\item[(a6)] If $0\le x_j \le y_j$ for every $j\in\Nat[d]$ and $(y_j)_{j=1}^d \in \TT_{A,B}$, for $A,B\subseteq \Nat[d]$ disjoint, then  $(x_j)_{j=1}^d \in \TT_{A,B}$.
\end{itemize}
Given $A\subseteq\Nat[d]$, put $\C_A=M_A(\TT_{\Nat[d]\setminus A,A})$. By definition,
\begin{itemize}
\item[(b1)] $-\C_A=\C_{\Nat[d]\setminus A}$.
\end{itemize}
We infer from (a1), (a2), (a3)  and (a6), respectively, that 
\begin{itemize}
\item[(b2)] if $\lambda$, $\mu\ge 0$ are such that $\lambda^p +\mu^p\le 1$, then $\lambda\, \C_A+\mu\, \C_A\subseteq \C_A$,  
\item[(b3)] $\{ (x_j)_{j=1}^d \in B_p \colon  \{j\colon x_j< 0\}=A \} \subseteq \C_A \subseteq B_\infty$,
\item[(b4)] $\C_\emptyset=B_p\cap \Rea^d_+$, and 
\item[(b5)] if $\sss\in [0,1]^d$ and $\xx\in \C_A$ then $M_\sss(\xx)\in\C_A$.
\end{itemize}
Consider the $d$-dimensional body
$
\C_{(p)}=\cup_{A\subseteq \Nat[d]} \C_A.
$
Properties (b2), (b3) and (b5)  give, respectively,
\begin{itemize}
\item[(c1)] $-\C_{(p)}=\C_{(p)}$, 
\item[(c2)] $B_p\subseteq \C_{(p)}\subseteq B_\infty$, and that
\item[(c3)] if $\xx \in \C_{(p)}$ and $\sss\in[0,1]^d$, then $M_\sss(\xx) \in \C_{(p)}$.
\end{itemize}
We infer from (a4) and (a5) that
\begin{itemize}
\item[(c4)] if $\xx=(x_j)_{j=1}^d\in \C_{(p)}$ and \[
\{ j\in\Nat[d] \colon x_j< 0\} \subseteq B\subseteq \{ j\in\Nat[d] \colon x_j\le 0\},
\] then $\xx \in \C_B$.
\end{itemize}
Combining (c4) with (b4) we obtain
\begin{itemize}
\item[(c5)]  $\C_{(p)}\cap \Rea^d_+=B_p\cap \Rea^d_+$.
\end{itemize}
Let us prove that $\C_{(p)}$ is absolutely $p$-convex.  Let  $\xx=(x_j)_{j=1}^d$,  $\yy=(y_j)_{j=1}^d  \in \C_{(p)}$ and $\lambda,\mu\in\Rea$ with $|\lambda|^p+|\mu|^p\le 1$. By (c1) we can assume that $\lambda,\mu\ge 0$. Let 
\begin{align*}
A & =\{ j \in \Nat[d] \colon \sgn(x_j)  \sgn(y_j)\not=-1\},\\
D & =\{ j \in \Nat[d]\setminus  A \colon \sgn( \lambda x_j + \mu y_j) = 0 \},\\
E & =\{ j \in \Nat[d]\setminus  A  \colon \sgn( \lambda x_j + \mu y_j) = \sgn(x_j) \},\\
F & =\{ j \in \Nat[d]\setminus  A  \colon \sgn( \lambda x_j + \mu y_j) = \sgn(y_j) \}.
\end{align*}
By construction $(A,D,E,F)$ is a partition of $\Nat[d]$. Note that $\lambda>0$ if $E\not=\emptyset$ and
$\mu>0$ if $F\not=\emptyset$. We define $\tilde\xx=(x_j)_{j=1}^d$,  $\tilde\yy=(y_j)_{j=1}^d$ and
 $\sss=(s_j)_{j=1}^d$ by 
\[
(\tilde x_j, \tilde y_j, s_j)=
\begin{cases}
(x_j,y_j,1) & \text{ if } j \in A, \\
(0,0,0)  & \text{ if } j \in D,\\
(x_j, 0,  (\lambda x_j)^{-1}(\lambda x_j + \mu y_j)) & \text{ if } j \in E,\\
(0, y_j, (\mu y_j)^{-1}( \lambda x_j + \mu y_j)) & \text{ if } j \in F.
\end{cases}
\]
By construction, $\sss\in[0,1]^d$ and $\lambda \xx + \mu \yy =M_\sss(\lambda \tilde \xx + \mu \tilde \yy )$.  Hence,
taking into account (c3), it suffices to prove that $\lambda \tilde \xx + \mu \tilde \yy\in\C_{(p)}$. Note that, by construction, $\sgn(\tilde x_j) \sgn(\tilde y_j)\not=-1$ for every $j\in\Nat[d]$. Therefore,  the set
$
\{ j\in\Nat[d] \colon \tilde x_j < 0 \} \cup \{ j\in\Nat[d] \colon \tilde y_j < 0 \}$ is contained in
\[ B:=\{ j\in\Nat[d] \colon \tilde x_j \le 0 \}\cap \{ j\in\Nat[d] \colon \tilde y_j \le 0 \}.
\]
Since, by (c3), $\tilde\xx$, $\tilde\yy\in\C_{(p)}$, we infer from (c4) that $\tilde\xx$, $\tilde \yy\in \C_B$. Then, by (b2), $\lambda \tilde \xx + \mu \tilde \yy\in\C_B\subseteq \C_{(p)}$.

Let $\Vert \cdot \Vert_{(p)}$ be the Minkowski functional associated to $\C_{(p)}$.  Taking into account (c2)  we infer that  $\Vert \cdot \Vert_{(p)}$ is a $p$-norm on $\Rea^d$.
By (c5), $\Vert \xx \Vert_{(p)} =\Vert \xx \Vert_p$ for every $\xx\in \Rea^d_+$.
\end{proof}
\begin{Proposition}\label{zeroone} Let $\MM$ be the  $\{0,1\}$-metric space and  $(a_x)_{x\in\MM\setminus\{0\}}$ be an eventually null family of scalars. Then
\[
 \left\Vert \sum_{x\in\MM\setminus\{0\}} a_x\, \delta(x) \right\Vert_{\F_p(\MM)} \ge 
 \left(\sum_{\substack{x\in\MM\setminus\{0\}\\ a_x\ge 0}} a_x^p\right)^{1/p}.
 \]
\end{Proposition}
\begin{proof}
Let $(a_x)_{x\in\MM\setminus\{0\}}\in[0,\infty)^{\MM\setminus\{0\}}$ be eventually null. Pick $d\in\Nat$ and a one-to-one map $\phi\colon \Nat[d] \to \MM\setminus\{0\}$  such that  
\[
 \{ x \in \MM\setminus\{0\} \colon a_x>0 \} \subseteq \phi(  \Nat[d] )\subseteq \{ x \in \MM\setminus\{0\} \colon a_x\ge 0 \}.
 \]
 Let $\Vert \cdot\Vert_{(p)}$ be as in Proposition~\ref{prop:pbody} and
consider the mapping $f\colon\MM\to ( \Rea^d,\Vert \cdot\Vert_{(p)})$ given by $\phi(k) \mapsto \ee_k$ for all $k\in \Nat[d]$ and 
 $x \mapsto 0$ if $x\notin \phi(  \Nat[d] )$. Since $\Vert\ee_i\Vert_{(p)}$, $\Vert \ee_i-\ee_j\Vert_{(p)}\le 1$ for every $i,j\in\Nat[d]$, $f$ is $1$-Lipschitz. Therefore
\begin{align*}
\left\Vert \sum_{x\in\MM\setminus\{0\}} a_x\, \delta(x) \right\Vert_{\F_p(\MM)}
&\ge  \left\Vert \sum_{x\in\MM\setminus\{0\}} a_x\, f (x) \right\Vert_{(p)}
= \left\Vert \sum_{k=1}^d  a_{\phi(k)} \ee_k\right\Vert_{(p)}\\
&= \left( \sum_{k=1}^d  a_{\phi(k)}^p\right)^{1/p}= \left(\sum_{\substack{x\in\MM\setminus\{0\}\\ a_x\ge 0}} a_x^p\right)^{1/p}.\qedhere
\end{align*}
\end{proof}

On occasion it will be convenient to know that the Lipschitz free $p$-space over a quasimetric space and  the Lipschitz free $p$-space over its completion are the same. Let us state this basic fact for reference and provide a proof using the tools that we introduced before.

\begin{Proposition}\label{density} Let $\MM$ be a  $p$-metric space for some $0<p\le 1$ and let $\NN$ be a dense subset of $\MM$ equipped with the same quasimetric.  Then 
$\F_p(\NN)\approx\F_p({\MM})$ isometrically. In fact, the canonical linear map is an isometry.
\end{Proposition}

\begin{proof} The canonical linear  map $L_{\jmath}\colon \PP({\NN})\to \PP({\MM})$  induced by the inclusion $\jmath$
from $\NN$ into ${\MM}$ is one-to-one on $\PP(\NN)$. By density, the set of molecules of $\MM$ of the form
\[
L_{\jmath}(\AAA(\NN))=\left\{\frac{\delta_{\MM}(y)-\delta_{\MM}(x)}{\rho(x,y)} \colon x,y\in \NN, \, x\not=y\right\}
\]
is an isometrically $p$-norming set in  $\F_{p}(\MM)$. Lemma~\ref{lemma:transformation} and Corollary~\ref{pNormingset}    yield that $L_{\jmath}$ extends to a linear isometry from $\F_p(\NN)$ onto $\F_p({\MM})$.
\end{proof}

We will study in detail some properties of the canonical map $L_{j}$  in Section~\ref{structureSec}. For the time being, to finish this section we provide  a sufficient condition for  $L_{\jmath}$ to be an isomorphic embedding.
 
\begin{Definition} Let $\MM$ be a $p$-metric space, $0<p\le 1$, and let $\NN$ be a subset of $\MM$. A Lipschitz map $r\colon \MM \to \NN$ is called a \textit{Lipschitz retraction} if it is the identity on $\NN$. When such a Lipschitz retraction exists we say that $\NN$ is a \textit{Lipschitz retract} of $\MM$.
\end{Definition}

\begin{Lemma}[cf.\ \cite{GodefroyKalton2003}*{Lemma 2.2}]\label{lemma:retract} 
Let $\MM$ be a  pointed $p$-metric space $(0<p\le 1)$ and $\NN$ be a Lipschitz retract of $\MM$. Then the inclusion map $\jmath\colon \NN \to \MM$ induces an isomorphic embedding $L_\jmath\colon \F_{p}(\NN)\to \F_{p}(\MM)$ onto a complemented subspace.
\end{Lemma}

\begin{proof} Without loss of generality we may and do assume that  $0\in \NN$. Let $\jmath\colon \NN \to \MM$ be the inclusion map  and let $r\colon\MM\to\NN$ be a Lipschitz retraction. Lemma~\ref{key} yields  $L_r\circ L_\jmath=\Id_{\F_{p}(\NN)}$, i.e., $L_\jmath\circ L_{r}$ is a linear projection from $\F_{p}(\MM)$ onto the linear subspace $L_{\jmath}(\F_{p}(\NN))$ of $\F_{p}(\MM)$  and $L_{\jmath}$ is an isomorphism.
\end{proof}

\subsection{Envelopes and duality}

\begin{Proposition}\label{envelopeoffree} Suppose $\MM$ is a pointed $p$-metric space with $q$-metric envelope $\widetilde{\MM^q}$, where $0<p<q\le 1$.  Then: 
\begin{enumerate}
\item[(a)] The $q$-Banach envelope of   $\F_{p}(\MM)$ is $\F_{q}(\widetilde{\MM^q})$.

\item[(b)]  In the particular case that $\MM$ is a $p$-Banach space $X$ with $q$-Banach envelope $\widehat{X^q}$, the $q$-Banach envelope of  $\F_{p}(X)$ is $\F_{q}(\widehat{X^q})$.

\end{enumerate}
\end{Proposition}
\begin{proof}The universal properties of  $q$-metric envelopes and  $q$-Banach envelopes yield the commutative diagram
\[
\xymatrix{{\MM}\ar[r]^Q \ar[d]_{\delta} & \widetilde{{\MM}^q}\ar[d]^{\delta}\\
\F_p({\MM}) \ar[r]^{L_Q} \ar[d]_{i}  &  \F_q(  \widetilde{{\MM}^q})\\
 \widehat{\F_p({\MM})^q} \ar[ur]_{\widetilde{L_Q}}   }
\]
Since co$_p(\AAA({\MM}))$ is dense in the unit ball $B$ of  $\F_p({\MM})$ and $i(\text{co}_q(B))$ is dense in the unit ball of $ X:=\widehat{\F_p({\MM})^q}$, we infer that  $\text{co}_q(i(\AAA({\MM})))$ is dense in the unit ball of $X$. Therefore, by Lemma~\ref{lemma:density}, $A:=i(\AAA({\MM}))$ is an isometrically $q$-norming set for $X$. Moreover, $\widetilde{L_Q}$ is a bijection from $i(\PP(\MM))$ onto $\PP( \widetilde{{\MM}^q})=\PP(Q(\MM))$ and $ \widetilde{L_Q} (A)=\AAA( \widetilde{{\MM}^q})=\AAA(Q(\MM))$. We deduce from Lemma~\ref{lemma:transformation} that   $\widetilde{L_Q}$  is an isometric isomorphism.
\end{proof}
 
\begin{Remark} The previous proposition implies, for example, that $\F_p(\Rea)$ is not isomorphic to $L_p$ for $p<1$. Indeed, its Banach envelope is $L_1$, while the Banach envelope of $L_p$ is trivial, and if two $p$-spaces are isomorphic, their envelopes are also isomorphic.\end{Remark}

\begin{Remark}  
Roughly speaking, it could be argued that given a $r$-metric space $\MM$, $0<r\le 1$, the family $(\F_{p}(\widetilde{\MM^p})) _{0< p\le 1}$,
where $\F_{p}(\widetilde{\MM^p})=\F_p(\MM)$ if $p\le r$, forms a scale of 
quasi-Banach spaces in the same way as the family $(\ell_p)_{0<p\le 1}$ does. Indeed, given $p\le q<1$, Proposition~\ref{envelopeoffree} provides  a canonical range-dense   linear map $L_{q,p}\colon \F_{p}(\widetilde{\MM^p}) \to  \F_{q}(\widetilde{\MM^q})$ with $\Vert L_{q,p}\Vert \le 1$ and, if $q<s\le 1$, we have $L_{s,p}=L_{s,q} \circ L_{q,p}$.

Let us restrict our attention to the case when $p<r$. Then
\[
L_{r,p}\colon \F_p(\MM)\to \F_r(\MM)
\]
is the identity map on $\PP(\MM)$ and, hence, it is one-to-one on a dense subspace. However we do not know if this map is always injective. In the case when $r=1$ we would like to point out that the map $L_{1,p}\colon \F_p(\MM)\to\F(\MM)$ is one-to-one if and only if $\F_{p}(\MM)^{\ast}$ separates the points of $\F_{p}(\MM)$.
\end{Remark}

\begin{Corollary}\label{dualAEp} Let $\MM$ be a pointed $p$-metric space, $0<p\le 1$. Then $\F_{p}(\MM)^{\ast} =\Lip_0(\MM)$, i.e.,  given $\phi\in \F_{p}(\MM)^{\ast}$ there is a unique $f\in \Lip_0(\MM)$ so that $\phi\big(\sum a_{i}\delta(x_{i})\big)  = \sum a_{i}f(x_{i})$ for every $\sum a_{i}\delta(x_{i})\in \F_{p}(\MM)$, and the map $\phi \mapsto f$ is a linear isometry of $\F_{p}(\MM)^{\ast}$ onto $\Lip_0(\MM)$. In particular,  $\F_{p}(\MM)^{\ast}=\{0\}$ if $\Lip_0(\MM)=\{0\}$.
\end{Corollary}
\begin{proof}  By identifying $\MM$ with $\delta(\MM)\subseteq \F_p(\MM)$ we get that the restriction of any $\phi\in\F_p(\MM)^*$ to $\MM$ belongs to $\Lip_0(\MM)$. And conversely, any $f\in\Lip_0(\MM)$ uniquely extends by the universal property to an element of $\F_p(\MM)^*$. This correspondence is a linear isometry.\end{proof}

 \begin{Corollary}\label{SecondC} Let $\MM$ and $\NN$ be  pointed metric spaces and suppose $0<p< 1$. If $\F_{p}(\MM)\approx \F_{p}(\NN)$ then $\F(\MM)\approx \F(\NN)$.
 \end{Corollary}
 \begin{proof} Just take Banach envelopes in Proposition~\ref{envelopeoffree} (a).
  \end{proof}

The last theorem of this section extends to the case when $0<p<1$ a result of Naor and Schechtman \cite{NaorSchechtman2007}.
\begin{Theorem}For any $0<p\le 1$, the $p$-Banach spaces $\F_{p}(\Rea)$ and $\F_{p}(\Rea^2)$ are not isomorphic.
\end{Theorem}
\begin{proof} The case $p=1$ was proved  in \cite{NaorSchechtman2007}. The case $0<p<1$ is taken care of by Corollary~\ref{SecondC}.
\end{proof} 
 
\section{Lipschitz free $p$-spaces over ultrametric spaces}\label{ultrametricSec}

\noindent The spaces $\F_{p}(\MM)$ over quasimetric (or even metric) spaces provide a new class of quasi-Banach spaces that  in general  are difficult to identify. The point of this section is to see that by imposing a stronger condition on  $\MM$, namely being ultrametric, we can recognize the Lipschitz free $p$-space over $\MM$. 

Recall that a distance $d$ on a set $\MM$ is called an \textit{ultrametric} provided that in place of the triangle inequality, $d$ satisfies the stronger condition 
\[
d(x,z)\le \max\{d(x,y), d(y,z)\},  \quad x,y,z\in\MM.
\]
Note that ultrametrics can be characterized as metrics $d$ such that $d^{\, p}$ is a metric for every
$p>0$.  Indeed, if $(\MM,d^{\, p})$ is a metric space for $p\in A$, and the set $A\subseteq \Rea$ is unbounded, then
\[
d(x,z)\le \left(d^{\, p}(x,y)+d^{\, p}(y,z)\right)^{1/p}, \quad x,y,z\in \MM,\, p\in A.
\]
Letting $p$ tend to infinity we get $d(x,z)\le \max\{ d(x,y),d(y,z)\}$. The converse implication is clear.

Before proceeding, let us digress a bit with the help of an example.  Let $(\MM,\le)$ be a totally ordered set and $\lambda=(\lambda_x)_{x\in \MM}$ a non-decreasing
family of positive numbers (it could also be non-increasing, in which case we would consider the reverse order on $\MM$). Equipped with $d_\lambda\colon\MM\times\MM\to[0,\infty)$ defined by
\[
d_\lambda(x,y)=\begin{cases}\lambda_{\max\{x,y\}} &\text{ if }\; x\not=y\\
0 &\text{ if }x=y,
\end{cases}
\]
$\MM$ is an ultrametric space. Since every set can be equipped with a total order, if we put $\lambda_x=1$ for all $x\in\MM$, we infer that  the $\{0,1\}$-metric on  $\MM$ and the ultrametric  $d_\lambda$ coincide.
 Thus the following theorem, which extends  to the case $p<1$ a result  from \cite{CuthDoucha2016}, also extends Theorem~\ref{discreteA} in the separable case.

\begin{Theorem}\label{MC2018} Let $(\MM,d)$ be an infinite separable pointed ultrametric space. Then $\F_p(\MM,d)\approx \ell_p$ for every $0<p\le 1$.
\end{Theorem}

The techniques we use to prove this theorem rely on the concepts of  $\Rea$-tree and length measure.  For the convenience of the reader we  include these definitions,  which we borrow from \cite{Godard2010}*{Sect.\ 2}. For more details concerning $\Rea$-trees see for instance \cite{Evans}*{Chapter 3}.

\begin{Definition}\label{Rtree} An \textit{$\Rea$-tree} is a metric space $(\T, d)$ satisfying:
\begin{enumerate}
\item[(i)] For any points $a$ and $b$ in $\T$, there exists a unique isometry $\phi$ from the closed interval $[0, d(a,b)]$ into $\T$ such that $\phi(0)=a$ and $\phi(d(a,b))=b$.
\item[(ii)] Any one-to-one continuous mapping $\varphi\colon [0,1]\to \T$ has the same range as the isometry $\phi$ associated to the points $a=\varphi(0)$ and $b=\varphi(1)$.

\end{enumerate}

\end{Definition}

If $\T$ is an $\Rea$-tree, given any $x$ and $y$ in $\T$ we denote by $\phi_{xy}$ the unique isometry associated to $x$ and $y$ as in Definition~\ref{Rtree}, and write $[x,y]$ for the range of  $\phi_{x,y}$. Such subsets of $\T$ are called \textit{segments}. Moreover, we say that $v\in \T$ is a \emph{branching point} of $\T$ if there are three points $x_1,x_2,x_3\in \T\setminus\{v\}$ such that $[x_i,v]\cap[x_j,v] = \{v\}$ whenever $i,j\in\{1,2,3\}$, $i\neq j$. We say that a subset $A$ of $\T$ is measurable whenever $\phi^{-1}_{xy}(A)$ is Lebesgue-measurable for any $x$ and $y$ in $\T$. If $A$ is measurable and $S$ is a segment $[x,y]$, we write $\lambda_{S}(A)$ for $\lambda(\phi_{xy}^{-1}(A))$, where $\lambda$ is the Lebesgue measure on $\Rea$. We denote by $\R$ the set of all subsets of $\T$ that can be written as a finite union of disjoint segments. For $R=\cup_{k=1}^{n}S_{k}$ (with disjoint $S_{k}$) in $\R$, we put
\[
\lambda_{R}(A)=\sum_{k=1}^{n}\lambda_{S_{k}}(A).
\]
Now,
\[
\lambda_{\T}(A)=\sup_{R\in \R}\lambda_{R}(A)
\]
defines a measure on the $\sigma$-algebra of $\T$-measurable sets called the \emph{length measure}. Note that this is nothing but the $1$-dimensional Hausdorff measure (multiplied by the constant $2$).

Suppose $(\Ss,d)$ is a closed subset of an $\Rea$-tree $\T$ with a base point $0\in\Ss$. For $s\in\Ss$ we put
\[L_\Ss(s):=\inf_{x\in [0,s)\cap {\Ss}} d(s,x).\]
If $L_\Ss(s) > 0$, we denote by $\sigma_\Ss(s)$ the unique point from $[0,s)\cap\Ss$ with $d(s,\sigma_\Ss(s)) = L(s)$. Finally, we put 
\[\Ss_{+}:=\{s\in\Ss\colon L_\Ss(s) > 0\}.\]

\begin{Lemma}\label{l:closedSubsetInTree}Let $(\Ss,d)$ be a closed subset of an $\Rea$-tree $\T$ with a point $0\in\Ss$. Let $\Ss$ have  length measure zero. Then for all $y\in \Ss$ and all $x\in [0,y]\cap \Ss$,
 \[
 d(x,y) = \sum_{z\in (x,y]\cap \Ss_{+}} L_\Ss(z).
 \]
\end{Lemma}
\begin{proof} Using the transformation $\phi_{0,y}$, we can assume without loss of generality that $x,y\in\Rea$, $0\le x \le y$ and $\Ss\subseteq[0,y]$. Then, the subset $(x,y)\setminus \Ss$ of $\Rea$ is open and, so, it can be expressed as $\cup_{i\in I} (a_i,b_i)$, where the intervals are disjoint. Since $\Ss$ has measure zero, we have $y-x=\sum_{i\in I} (b_i-a_i)$. It is clear that $b_i\in \Ss_{+}$ and $\sigma_\Ss(b_i)=a_i$ for every $i\in I$. Thus, it suffices to see that the map $b\colon I \to \Ss_{+}$, $i\mapsto b_i$ is onto. Given $s\in \Ss_{+}$, since $(\sigma_\Ss(s),s)\cap\Ss=\emptyset$, there exists $i\in I$ such that
$(\sigma_\Ss(s),s)\subseteq(a_i,b_i)$. Taking into account that neither $\sigma_\Ss(s)$ nor $s$ belong to $(a_i,b_i)$, we infer that $a_i=\sigma_\Ss(s)$ and $b_i=s$.
\end{proof}
In the case when $p=1$ the following result was proved by Godard (see \cite{Godard2010}*{Proposition 2.3}). Below we give an alternative proof which works for every $0<p\leq 1$ and for not necessarily separable $\Rea$-trees $\T$.
\begin{Proposition}
\label{prop:godardAnalogyMeasureZero}Let $(\Ss,d)$ be a closed subset of an $\Rea$-tree $\T$ such that $\Ss$ contains all the branching points of $\T$ and  has length measure zero. Then $\F_p(\Ss,d^{1/p})\approx \ell_p(\Ss_{+})$ isometrically. To be precise, the map
\[
T(\delta(s)) := \sum_{x\in (0,s]\cap \Ss_{+}} (L_\Ss(x))^{1/p}\ee_x,\qquad s\in \Ss,
\]
extends to a linear isometry between $\F_p(\Ss,d^{1/p})$ and $\ell_p(\Ss_{+})$.
\end{Proposition}
\begin{proof}
Without loss of generality we assume that $0\in\Ss$, and for simplicity, for each $s\in\Ss_{+}$ we denote $\sigma_\Ss(s)$ by $s^\#$.
For every $y\in\Ss$ and $x\in[0,y]\cap \Ss$ we have
\begin{equation}\label{eq:pathDelta}
\delta(y) - \delta(x) = \sum_{z\in (x,y]\cap \Ss_{+}} \big(\delta(z) - \delta(z^\#)\big).
\end{equation}
To see this,  if we  consider in $[x,y]$ the total order induced by the isometry $\phi_{x,y}$,
by Lemma \ref{l:closedSubsetInTree} for any $\varepsilon>0$ we can get $z_1<z_2<\cdots<z_n\in (x,y]\cap \Ss_{+}$ with
$
D:= \left|\sum_{i=1}^n d(z_i,z_i^\#) - d(x,y)\right|\leq \epsilon.
$
With the convention that $z_0=x$,
\begin{align*}
\Big\Vert\delta(y) - \delta(x) &- \sum_{i=1}^n (\delta(z_i) - \delta(z_i^\#))\Big\Vert^p  \\
&\le \|\delta(y) - \delta(z_n)\|^p + \sum_{i=1}^n \|\delta(z_i^\#) - \delta(z_{i-1})\|^p \\
&\le d(y,z_n)+\sum_{i=1}^nd(z_i^\#, z_{i-1})
=D \leq \varepsilon,
\end{align*}
where, in the last equality we used the fact that $z_{i-1}\le z_i^\#$ for every $i=1,\dots,n$.
For $s\in \Ss_{+}$, put
\[
\aaa_s:= \frac{\delta(s) - \delta(s^\#)}{d^{1/p}(s,s^\#)},
\]
and set $
\AAA:=\overline{\cop} \left\{\aaa_s\colon s\in \Ss_{+}\right\}.
$
Note that $T(\aaa_s) = \ee_s$  since $(s^\#,s)\cap\Ss=\emptyset$. By Lemma~\ref{lemma:transformation}, in order to show that $T$ extends to a surjective isometry it suffices to show that for all $x,y\in \Ss$,
\begin{equation}\label{diseq}
a_{x,y}:=\frac{\delta(x) - \delta(y)}{d^{1/p}(x,y)} \in \AAA,
\end{equation} 
 To that end, given $x,y\in \Ss$ there are three cases to take into account:

\noindent$\bullet$ {\sc Case 1:}  If $x\in[0,y]$, then by the identity \eqref{eq:pathDelta}, 
\[
a_{x,y}  = \sum_{z\in (x,y]\cap \Ss_{+}} \frac{(L_{\Ss}(z))^{1/p}}{d^{1/p}(x,y)}\, \aaa_z,
\]
which, by Lemma~\ref{l:closedSubsetInTree}, is a $p$-convex combination of $\aaa_z$, and so $a_{x,y}\in \AAA$.

\noindent $\bullet$ {\sc Case 2:} If $y\in[0,x]$, switching the roles of $x$ and $y$ in the previous case we easily get $a_{x,y}\in \AAA$.

\noindent $\bullet$ {\sc Case 3:} If neither Case 1 nor Case 2 occurs, there exists a branching point $c\in[x,y]$ with $[0,x]\cap [0,y] = [0,c]$. Then we can write
\[
a_{x,y}=\lambda a_{y,c}+\mu a_{c,x}, \quad \lambda=\frac{d^{1/p}(y,c) }{d^{1/p}(x,y) }, \quad
\mu= \frac{d^{1/p}(c,x)}{d^{1/p}(x,y) }.
\]
From the two previous  cases we have $a_{c,y}$, $a_{c,x}\in \AAA$, and since 
\[
\lambda^p+\mu^p =\frac{d(x,c)+d(c,y)}{d(x,y)}=1,
\]
we conclude that $a_{x,y}\in \AAA$. Hence, \eqref{diseq} is fulfilled.
\end{proof}

Since the real line is a trivial example of an $\Rea$-tree we obtain:
\begin{Corollary}\label{cor:measureZero}
\label{AnsoTh}  Let $\MM$ be an infinite subset of $\Rea$ and  $0<p\le 1$.  If the closure of  $\MM$  has measure zero then
$
\F_p(\MM,|\cdot|^{1/p}) \approx \ell_p$ isometrically.
In particular, the result holds if $\MM$ is the range of a monotone sequence of real numbers.
\end{Corollary}

\begin{proof}[Proof of Theorem~\ref{MC2018}] Since $d^{\, p}$ is also an ultrametric whenever $d$ is, we need only show that $\F_p(\MM,d^{1/p})\approx \ell_p$. By \cite{CuthDoucha2016}*{Proposition 12}, there exists a closed subset $\Ss$  of  a separable $\Rea$-tree $\T$  containing all its branching points in such a way that $\Ss$ has length measure zero and $(\MM,d)$ is bi-Lipschitz isomorphic to a Lipschitz retract of $\Ss$. Denoting the metric on $\Ss$ by $\eta$, we have that $(\MM,d^{1/p})$ is bi-Lipschitz isomorphic to a Lipschitz retract of $(\Ss,\eta^{1/p})$. By Proposition~\ref{prop:godardAnalogyMeasureZero}, Lemma~\ref{lem:oodim}, and Corollary~\ref{cor:sep}, $\F_p(\Ss,\eta^{1/p})\approx \ell_p$ hence, by Lemma~\ref{lem:oodim} and Lemma~\ref{lemma:retract}, $\F_p(\MM,d^{1/p})$ is isomorphic to an infinite-dimensional complemented subspace of $\ell_{p}$. Since  every infinite-dimensional complemented subspace of $\ell_p$ is isomorphic to $\ell_{p}$ by a classical result of Stiles  \cite{Stiles1972}, we infer that $\F_p((\MM,d^{1/p}))\approx \ell_p$, and the proof is over.
\end{proof}

Theorem~\ref{LpasFreespace} and Corollary~\ref{AnsoTh} allow us to identify the free $p$-spaces over some subsets of the real line equipped with the
``anti-snowflaking'' quasimetric $|\cdot|^{1/p}$. However, identifying  the free $p$-space over subsets of $\Rea$ equipped with the Euclidean distance seems to be a more challenging task.  For the time being, let us just mention that since the Banach envelope of $\F_p(I)$ is $L_1(I)$ for any interval $I$ with the Euclidean distance (to see this, apply Proposition~\ref{envelopeoffree}(b) and  Theorem~\ref{LpasFreespace} for $p=1$), the spaces $\F_p(I)$ constitute   a new class of $p$-Banach spaces.  

\section{Linearizations of Lipschitz embeddings}\label{structureSec}

\noindent Lipschitz free $p$-spaces over quasimetric spaces  constitute a nice family of new $p$-Banach spaces which are easy to define but whose geometry seems to be difficult to understand. To carry out this task  sucessfuly one hopes to be able to count on ``natural''  structural results involving free $p$-spaces over subsets of $\MM$.  In this section we analyse this premise and confirm an unfortunate recurrent pattern in quasi-Banach spaces: the lack of tools can be an important stumbling block in the development of the  nonlinear theory. However, as we will also see, not everything is lost and we still can develop  specific methods that permit to shed light onto the structure of $\F_{p}(\MM)$.

 If $\MM$ is a pointed $p$-metric space and  $\NN$ is a subset of $\MM$ containing $0$, the
linearization process of Lemma~\ref{key} applies in particular to the canonical injection $\jmath\colon \NN\to \MM$.
If $p=1$,   McShane's theorem \cite{McShane1934} ensures that $L_\jmath\colon \F(\NN)\to \F(\MM)$ is an isometric embedding.
Thus $\F(\NN)$ can be naturally identified with a subspace of $\F(\MM).$  However, if $p<1$ this argument crashes.    
In the case when $p=1$, $\F(\NN)$ is isometric to a subspace of $\F(\MM)$ and so the study of  Lipschitz free spaces over subsets is a powerful tool.  We start by exhibiting that this argument breaks down when $p<1$, settling a problem raised in \cite{AlbiacKalton2009}.

\begin{Theorem}
For each $0<p<1$ and $p\leq q\leq 1$ there is a $q$-metric space  $(\MM,\rho)$ and a subset $\NN\subseteq\MM$ such that the inclusion map $\jmath\colon\NN\to\MM$ induces a non-isometric isomorphic embedding $L_\jmath:\F_p(\mathcal N)\to \F_p(\mathcal M)$ with $\Vert L_\jmath^{-1}\Vert \ge 2^{1/q}$.
\end{Theorem}
\begin{proof}
Let $\NN=\{0\}\cup\Nat$  and $\MM=\{0,z\}\cup\Nat$, where $z\notin\NN$. Define 
$\rho\colon\MM \times\MM\to[0,\infty)$ by $\rho(x,x)=0$  for all $x\in\MM$, $
\rho(j,z)=\rho(z,j)=2^{-1/q}$  for all $ j\in\Nat$, and $\rho(x,y)=1$ otherwise.
It is clear that $(\MM,\rho)$ is a bounded and uniformly separated $q$-metric space. Then, by Theorem~\ref{discreteA}, $L_\jmath$ is an isomorphism.

Given $\epsilon>0$, there is $k\in\Nat$ and a $k$-tuple $(a_j)_{j=1}^k$ such that $a_j\ge 0$ for every $j=1$, \dots, $k$, $\sum_{j=1}^k a_j^p=1$, and 
$\sum_{j=1}^k a_j\le \epsilon$. Indeed,  it suffices to choose $k\geq \varepsilon^{p/(p-1)}$  and put $a_j= k^{-1/p}$ for $1\le j \le k$. On the one hand, since $\rho$ is the $\{0,1\}$-metric on $\NN$, by Proposition~\ref{zeroone}, $
\| \sum_{j=1}^k a_j \, \delta(j) \|_{\F_p(\NN)}=1.$
On the other hand, considering the decomposition
\[
 \sum_{j=1}^k a_j \, \delta(j) = \left(\sum_{j=1}^k a_j\right) \delta(z)+ \sum_{j=1}^k \frac{a_j}{2^{1/q}} \, 
\frac{\delta(j) -\delta(z)}{2^{-1/q}}
\]
and using Corollary~\ref{alternative-norm}, we have
\[
\left\| \sum_{j=1}^k a_j \, \delta(j) \right\|_{\F_p(\MM)}^p\le  \left(\sum_{j=1}^k a_j\right)^p + \sum_{j=1}^k (2^{-1/q}a_j)^p
\le \epsilon^p+2^{-p/q}.
\]
Hence,
\[
\Vert L_\jmath^{-1}\Vert \ge \sup_{\epsilon>0} \frac{1}{ (\epsilon^p+2^{-p/q})^{1/p}}=2^{1/q}.\qedhere
\]
\end{proof}
The following problem seems to be open.
\begin{Question}\label{QA1}
Let $0<p<1$ and $\NN\subseteq \MM$ be two $p$-metric (or metric) spaces in inclusion.  Is the canonical linear map of $\F_p(\NN)$ into $\F_p(\MM)$ an isomorphic embedding?
\end{Question}
 
The answer to Question~\ref{QA1} is positive in some special cases. 

\begin{Proposition} If $\MM$ is a $\{0,1\}$-metric space, 
then the canonical map of $\F_p(\NN)$ into $\F_p(\MM)$ is isometric for every $\NN\subseteq\MM$. 
\end{Proposition}
\begin{proof}
Notice that a map $f$ from a  $\{0,1\}$-metric space $\MM$ into a quasi-Banach space $X$ is $1$-Lipschitz if and only if 
$\Vert f(x)-f(y)\Vert\le 1$  for every $x$, $y\in\MM$. Then, a  $1$-Lipschitz map $f\colon\NN\to X$ with $f(0)=0$ extends by $f(x)=0$ for $x\in \MM\setminus\NN$
to a $1$-Lipschitz map from $\MM$ into $X$. This gives $\Vert L_\jmath(\mu)\Vert_{\F_p(\MM)}\ge \Vert \mu\Vert_{\F_p(\NN)}$ for every $\mu\in\PP(\NN)$.
\end{proof}

\begin{Proposition}Let $\MM$ be a subset of $\Rea$ equipped with the quasimetric 
$\rho(x,y)=|x-y|^{1/p}$ for  $0<p\le 1$.  If the closure of  $\MM$  has measure zero and $\NN\subseteq \MM$ then the canonical mapping $L_\jmath\colon \F_{p}(\NN)\to  \F_{p}(\MM)$ is an isometry.
\end{Proposition}
\begin{proof} By Proposition~\ref{density} we can assume that both $\NN$ and $\MM$ are closed. Let 
 $T_\MM\colon\F_p(\MM)  \to \ell_p(\MM_{+})$ and $T_\NN\colon \F_p(\NN) \to \ell_p(\NN_{+})$ be the isomotries provided by Proposition~\ref{prop:godardAnalogyMeasureZero}. If $U=T_\MM\circ L_\jmath \circ T_\NN^{-1}$, where $\jmath\colon\NN\to\MM$ is the inclusion map,  we have
\[
U(\ee_s)=\frac{1}{| s -\sigma_\NN(s)|^{1/p}}\sum_{z\in(\sigma_\NN(s),s]\cap \MM_{+}} L_{\MM}^{1/p}(z) \ee_z, \quad s\in\NN_{+}.
\]
By Lemma~\ref{l:closedSubsetInTree}, $\Vert U(\ee_s)\Vert_p=1$ for every $s\in\NN_{+}$. Pick $s$, $t\in\NN_{+}$ with $s<t$. Since $(\sigma_\NN(s),s)\cap\NN=\emptyset$, we have $s\le\sigma_\NN(t)$. Then, $(\sigma_\NN(s),s]\cap (\sigma_\NN(t),t]=\emptyset$ and so $(U(\ee_s))_{s\in\NN_{+}}$ is a disjointly supported family in $\ell_p(\MM_{+})$. Hence
$(U(\ee_s))_{s\in\NN_{+}}$ is isometrically equivalent to the unit vector basis of $\ell_p(\MM_{+})$, i.e., $U$ is an isometric embedding.
\end{proof}

\begin{Question}
Let $0<p<1$. Can we identify the $p$-metric spaces $\MM$ for which the canonical linear map from $\F_p(\NN)$ into $\F_p(\MM)$ is an isometry  for every $\NN\subseteq \MM$?
\end{Question}

\subsection*{Acknowledgment} 
F. Albiac acknowledges the support of the  Spanish Ministry for Economy and Competitivity Grants MTM2014-53009-P  for \textit{An\'alisis Vectorial, Multilineal y Aplicaciones},  and MTM2016-76808-P for  \textit{Operators, lattices, and structure of Banach spaces}.  J.L. Ansorena  acknowledges the support of the  Spanish Ministry for Economy and Competitivity Grant MTM2014-53009-P for  \textit{An\'alisis Vectorial, Multilineal y Aplicaciones} M. C\'uth has been supported by Charles University Research program No. UNCE/SCI/023 and by the Research grant GA\v{C}R 17-04197Y. M. Doucha was supported by the GA\v CR project 16-34860L and RVO: 67985840.

F. Albiac and J.L. Ansorena would like to thank the Faculty of Mathematics and Physics at Charles University in Prague for their hospitality and generosity during their visit in September 2018, when most work on this paper was undertaken.

\end{document}